\documentclass[a4paper,oneside,10pt,,mathscr,final,reqno]{amsart}

\usepackage{amssymb}
\usepackage{euscript}
\usepackage[arrow,line,matrix]{xy}
\usepackage{version}
\usepackage{color}
\usepackage{showkeys}

 \excludeversion{NB}

\newcommand{\bb} {\mathbb}
\newcommand{\cal}{\mathcal}
\newcommand{\frk}{\mathfrak}

\newcommand{\cA}{\cal{A}}
\newcommand{\cC}{\cal{C}}
\newcommand{\fg}{\frk{g}}
\newcommand{\id}{\operatorname{id}}
\newcommand{\ve}{\varepsilon}
\newcommand{\wt}{\widetilde}
\newcommand{\vac}{\left|\, 0\, \right>}
\newcommand{\ghat}{\widehat{\fg}}

\newcommand{\FG}{H_a}
\newcommand{\Lg}{L\fg}
\newcommand{\Ob}{\operatorname{Ob}}
\newcommand{\RM}{R\mathrm{-Mod}}
\newcommand{\CM}{\mathbb{C}\mathrm{-Mod}}
\newcommand{\Se}{S_0} 
\newcommand{\SY}{S_t}

\newcommand{\AHS}{(\cA,H,S)}
\newcommand{\MHS}{(\RM,H_m,S_L)}
\newcommand{\RHS}{(\RM,\FG,\Se)}
\newcommand{\Aut}{\operatorname{Aut}}
\newcommand{\Der}{\operatorname{Der}}
\newcommand{\End}{\operatorname{End}}
\newcommand{\Fin}{\mathrm{Fin}}
\newcommand{\Fun}{\operatorname{Fun}}
\newcommand{\Hom}{\operatorname{Hom}}
\newcommand{\Ind}{\operatorname{Ind}}
\newcommand{\Lie}{\operatorname{Lie}}
\newcommand{\Res}{\operatorname{Res}}
\newcommand{\Vir}{\mathrm{Vir}}

\newcommand{\Fine}{\Fin^{\not\equiv}}
\newcommand{\QCoh}{\operatorname{QCoh}}

\newcommand{\Vd}{V_{\delta}}
\newcommand{\VLg}{V_{\Lg}}
\newcommand{\Vvir}{V_{Vir, c}}
\newcommand{\Vghat}{V_{\ghat, k}}

\newcommand{\dd}{\oalign{\raisebox{0.4em}[0.5em][0em]{$\circ$} \cr
                         \raisebox{0.15em}[0.5em][0em]{$\circ$}}}
\newcommand{\nord}[1]{\dd #1 \dd}

\newcommand{\simto}{\xrightarrow{\, \sim \,}}
\newcommand{\longto}{\longrightarrow}
\newcommand{\longinto}{\lhook\joinrel\longrightarrow}

\theoremstyle{plain}
 \newtheorem{thm}{Theorem}[section]
 \newtheorem{lem}[thm]{Lemma}
 \newtheorem{prop}[thm]{Proposition}
 \newtheorem{cor}[thm]{Corollary}
 \newtheorem*{thm*}{Theorem}
\theoremstyle{definition}
 \newtheorem{dfn}[thm]{Definition}
 \newtheorem{eg}[thm]{Example}
 \newtheorem{fct}[thm]{Fact}
 \newtheorem{rmk}[thm]{Remark}
\theoremstyle{remark}

\numberwithin{equation}{section}
\allowdisplaybreaks

\begin{document}


\title{Classical and quantum conformal field theories}
\author{Shintarou Yanagida}
\address{Research Institute for Mathematical Sciences,
Kyoto University, Kyoto 606-8502, Japan}
\email{yanagida@kurims.kyoto-u.ac.jp}

\thanks{
This work is partially supported by the 
JSPS Strategic Young Researcher 
Overseas Visits Program for Accelerating Brain Circulation
``Deepening and Evolution of Mathematics and Physics,
Building of International Network Hub based on OCAMI''}

\date{November 30, 2013}


\begin{abstract}
Following the formuation of Borcherds,
we develop the theory of (quantum) $\AHS$-vertex algebras,
including several concrete examples.
We also investigate the relationship between 
the $\AHS$-vertex algebra 
and the chiral algebra due to Beilinson and Drinfeld.
\end{abstract}

\maketitle


\section{Introduction}
The notion of vertex algebra \cite{B:1986} was introduced by 
Borchreds for a formulation of two-dimensional conformal field theory. 
Although this formulation is successful for encoding algebraic structure of conformal field theory 
and giving representation theoretic treatment,
it involves somewhat complicated axioms and lacks geometric interpretation 
of the quantum field theory.

In \cite{B:1998} and \cite{B:2001},
Borcherds himself reformulated the axiom of vertex algebras 
and built the theory of $\AHS$-vertex algebras.
One of the motivation of his reformulation was the simplification (trivialization) 
of the axioms of vertex algebras,
Another motivation was to relate deformations of vertex algebras with the non-deformed 
vertex algebras in a simple way.

In this note, following the formuation of Borcherds,
we develop the theory of $\AHS$-vertex algebra 
and its quantization,
including several concrete examples.
We also investigate the relationship between 
the vertex algebra (in our sense) 
and the chiral algebra due to Beilinson and Drinfeld.
We shall show that an $\AHS$-vertex algebra 
in a geometric setting gives the reformulation of 
the chiral algebra.
Since the notion of chiral algebra has geometric flavor,
we may say that the $\AHS$-vertex algebra help us in 
geometric investigation of conformal field theory and 
its quantum deformations.

Let us explain the organization of this note briefly.
In \S 1, we review the theory of Borcherds' $\AHS$-vertex algebras,
so the readers who are familiar with the discussion in \cite{B:1998} and 
\cite{B:2001} may skip the details in this part.
Let us mention that in \S 1.3 we included a slightly generalized treatment of 
twisting construction.
In \S 2 we give a few examples of $\AHS$-vertex algebras.
In \S 3 we recall the notion of quantum $\AHS$-vertex algebras.
Its relationship with the deformed chiral algebra is stated in \S 3.3.
In the final \S 4 we investigate the relationship between $\AHS$-vertex algebras 
and the Beilinson-Drinfeld chiral algebras. 

Let us fix some global notations in this note.
\begin{itemize}
\item
For a category $\cC$, the class of objects is denoted by $\Ob(\cC)$ or $\Ob \cC$.
and the class of morphisms between objects $A,B$ is denoted by $\cC(A,B)$ 
or $\Hom_{\cC}(A,B)$.
\item
The composition of morphisms $f:A \to B$ and $g: B \to C$ 
is denoted by $g \circ f$.
\item
Functors between categories means covariant functors.
\item
For a category $\cC$,
its opposite category is denoted by $\cC^{op}$.
\item
For a bialgebra $B$ over a commutative ring $R$ 
let us denote by $\Delta_B$ and $\ve_B$ 
the comultiplication $B \otimes_R B \to B$ 
and the counit $B \to R$. 
\item
For an element $a$ of a bialgebra $B$ 
we express the comultiplication of $a$ by 
$$
 \Delta_B(a)=\sum a' \otimes a''=\sum_{(a)} a' \otimes a'' 
  =\sum_{(a)} a^{(1)} \otimes a^{(2)}.
$$
\end{itemize}

The vertex algebra in the sense of \cite{B:1986} will be called 
ordinary vertex algebra 
(see Definition \ref{dfn:ova} for the precise definition).

\section{Borcherds' formulation of $(\cA,H,S)$-vertex algebras}

In this section we review the formulation of $(\cA,H,S)$-vertex algebra 
due to Borcherds \cite{B:2001}, which is a generalization (and simplification) 
of the classical axiom of vertex algebras \cite{B:1986}.

After recalling the categorical treatment in \cite{B:2001},
we will answer a problem stated there \cite[\S5, Problem 5.5]{B:2001}:
Construct $(R\text{-mod},H,S)$ vertex algebras 
corresponding to the other standard examples of vertex algebras,
such as the vertex algebras of affine and Virasoro algebras.

\subsection{Categorical setting}
\label{subsec:cat}

\begin{dfn}
\begin{enumerate}
\item
Consider a category 
whose objects are finite sets 
and whose morphisms are  arbitrary maps between them.
Denote its skeleton by $\Fin$.
\item
Consider a category 
whose objects are finite sets 
and an equivalence relation $\equiv$,
and whose morphisms are the maps $f$ preserving inequivalence,
i.e.,  we have $a \equiv b$ if $f(a)=f(b)$.
Denote its skeleton by $\Fine$. 
\end{enumerate}
\end{dfn}

Note that both $\Fin$ and $\Fine$ are small.

Although these categories are defined as skeletons of some other categories
and objects should be called `isomorphic classes of sets',
we call them just by `sets' for simplicity. 

Objects of $\Fin$ will be expressed 
by $\emptyset,\{1\}, \{1,2\}, \{1,2,3\}, \ldots$, in form of finite sets.
We will also use the symbols $\{2\}$, $\{3\}$ 
for plain explanations in the discussion later,
although these objects are the same as $\{1\}$.

When denoting an object of $\Fine$,
we will use colons to separate equivalence classes.
For example, $\{1;2\}$ means 
a set consisting of two objects with two equivalent classes,
and $\{1,2\}$ means 
a set consisting of two objects with one  equivalent class.

The disjoint union is a coproduct on the category $\Fin$,
and it makes $\Fin$ into a symmetric monoidal category
(in the sense of \cite{M}). 
We denote the disjoint union in $\Fin$ by the symbol $\sqcup$.
We may define an analogue of the disjoint union for $\Fine$ as follows.

\begin{dfn}
For objects $I$ and $J$ in $\Fine$, 
we define $I \sqcup J$ 
to be the disjoint union of $I$ and $J$ as sets 
with the equivalence relation where 
an element of $I$ and another of $J$ are inequivalent   
and the other cases are 
determiend by the equivalence relations in $I$ and $J$.  
We call this $\sqcup$ on $\Fine$ simply by disjoint union.
\end{dfn}

Then the disjoint union $\sqcup$ on $\Fine$ 
gives a symmetric monoidal structure on $\Fine$,
although it is not a coproduct on $\Fine$ 
as mentioned in \cite[\S3]{B:2001}.f

Also note that $\Fin$ can be considered as a full subcategory of $\Fine$
by imposing the indiscrete equivalence on each set
(all the elements in a set are defined to be equivalent).
This embedding is denoted by 
\begin{align}
\label{eq:Fin:embed}
\iota: \Fin \longinto \Fine
\end{align}
 
In the following we fix a category $\cA$  
which is additive, symmetric monoidal, 
cocomplete and 
such that colimits commute with tensor products.
We denote by $\otimes$ 
the bifunctor $\cA \times \cA \to \cA$ 
giving the monoidal structure of $\cA$,
and by ${\bf 1}$ the unit object.
When emphasizing that we are considering the monoidal category $\cA$,
we sometimes denote the tensor product by $\otimes_{\cA}$.
The isomorphism $M \otimes_{\cA} N \to N \otimes_{\cA} M$ 
giving the symmetric monoidal structure on $\cA$ 
will be denoted by $\sigma_{M,N}$
and called symmetry.

The main example of $\cA$ we consider is the category 
$\RM$ of modules over a commutative ring $R$.
The tensor product is given by the tensor product $\otimes_R$ 
of modules over $R$,
and the symmetry is given by the transposition map 
$\sigma_{M,N}: M\otimes_R N \to N \otimes_R M$ 
of $R$-modules.

\begin{dfn}
For a  category $\cC$ 
let us denote by $\Fun(\cC,\cA)$ 
the category of functors from $\cC$ to $\cA$. 
By the additive monoidal structure on $\cA$,
the category $\Fun(\cC,\cA)$ is an additive symmetric monoidal structure,
where the tensor product is given by 
$(U \otimes V)(I) := U(I) \otimes_{\cA} V(I)$ 
for $I \in \Ob(\cA)$ and $U,V \in \Ob(\Fun(\cC,\cA))$. 
\end{dfn}

Let us recall the notion of rings (or algebras) in monoidal categories.
A ring object $A$ of a monoidal category $(\cal{D},\otimes,{\bf 1})$ is 
an object of $\cal{D}$ such that 
for any $X \in \Ob(\cal{D})$ 
the set of morphisms $\Hom_{\cal{D}}(X,A)$ is a ring,
and the correspondence $Y \to \Hom_{\cal{D}}(X,A)$ 
is a functor from $\cal{D}$ to the cageory of rings.
Here a ring means an associative unital ring.

If $\cal{D}$ has finite products and a terminal object $T$,
then a ring object can be defined similarly as the usual ring:
there exist morphisms $a: A \otimes A \to A$ (addition), 
$r: A \to A$ (inversion),
$z: T \to A$ (zero),
$m: A \otimes A \to A$ (multiplication)
and $u: {\bf 1} \to A$ (unit),
satisfying the sets of axioms.
%

One can define a commutative ring object as 
a ring object with the multiplication $m$ satisfying the commutative axiom.
We omit the detail.

A coalgebra object is defined in a similar way,
as an object with morphisms 
$a,r,z$, $\Delta:A \to A \otimes A$ (comultiplication) 
and $\ve: A \to  {\bf 1}$ (counit) 
satisfying several sets of axioms.
A cocommutative coalgebra object is defined in a similar way.

Similarly we can define a bialgebra object,
a module object over a ring object,
a comodule object over a coalgebra object 
and so on in a given category.

Hereafter the symbol $\cC$ 
means the category $\Fin$ or $\Fine$.

\begin{dfn}
\label{dfn:T_*}
Let $\cC$ be the category $\Fin$ or $\Fine$.
Let $A$ be a ring object in $\cA$.
Define an object $T_*(A)$ in $\Fun(\cC,\cA)$ 
by $T_*(A)(I) := \otimes_{i \in I} A$ for $I \in \Ob(\cC)$,
and for a morphism $f: I \to J$ in $\cC$ 
define $T_*(A)(f): T_*(A)(I)  \to T_*(A)(J)$ 
in a natural way by the multiplication and the unit of $A$. 
We sometimes write $f_* := T_*(A)(f)$ for simplicity.
\end{dfn}

Let us expain the `natural way' in the definition above by examples.

\begin{eg}
\label{eg:f}
\begin{enumerate}
\item
For the identity morphism $\id_I: I = \{1,2,\ldots,n\} \to I$ in $\Fin$,
$\id_{I,*}: A^{\otimes n} \to A^{\otimes n}$ is given by $\id_{A^{\otimes n}}$.

\item 
Consider a surjective morphism $p: \{1,2\} \to \{1\}$ in $\Fin$.
Then $p_*: A \otimes_{\cA} A  \to A$
is defined to be the multiplication morphism 
$m: A \otimes_{\cA} A \to A$ of $A$.

\noindent
For the morphism 
$p: \{1,2,3\} \to \{1\}$ in $\Fin$, 
$p_*: A \otimes_{\cA} A \otimes_{\cA} A  \to A$
is defined to be the composition of multiplication morphisms 
$m^2 := m \circ (m \otimes \id_{\cA})= m \circ ( \id_{\cA}  \otimes m )$.

\noindent
In general, for the surjective morphism 
$p_I: I = \{1,2,\ldots,n\} \to \{1\}$ in $\Fin$,
$p_{I,*}$ is defined by the $n$-times composition of multiplication morphisms

\item
For the morphism  
$i: \emptyset \to \{1\}$,
$i_*: T \to A$ 
is defined to be the unit morphism $u:T \to A$.

\noindent
Similarly, for the morphism 
$i_I: \emptyset \to I = \{1,2,\ldots,n\}$,
$i_{I,*}: T \to A^{\otimes n}$ 
is given by  $u^{\otimes n}$.

\item
For the morphism  $s: \{1,2\} \to \{1,2\}$ in $\Fin$ 
with $s(1)=2$ and $s(2)=1$,
$s_*: A \otimes_{\cA} A  \to A \otimes_{\cA} A$
is defined to be the isomorphism 
$\sigma_{A,A}: A \otimes_{\cA} A  \to A \otimes_{\cA} A$  
given by the symmetric monoidal structure of $\cA$.
\end{enumerate}
\end{eg}

Since any morphism $f$ in $\Fin$ can be decomposed into $\id_I$, $p_J$ and 
$s$ given in Example above,
we can compute $f_*$ by combining the rules given above.
The ways of decomposition are not unique,
but the resulting $f_*$ is determined uniquely 
by the symmetric monoidall structure of $\cA$.

Here we give a few more examples for $f_*$.

\begin{eg}
\begin{enumerate}
\item
For the morphism  
$i: \{1\} \to \{1,2\}$ with $i(1)=1$,
$i_*: A \to A \otimes A$ 
is defined to be $(\id_A, u \circ t_{A})$,
where $t_A: A \to T$ is the canonical 
morphism from $A$ to the terminal object $T$.

\noindent
In general,
for the injective morphism 
$i: \{1,2,\ldots,m\} \to \{1,2,\ldots,n\}$ ($m \le n$) 
with $i(j)=j$,
$i_*: A^{\otimes m} \to A^{\otimes n}$
is defined to be 
$\id_A^{\otimes m} \otimes  u^{\otimes(n-m)}$.

\item
For the morphism $f: \{1,2\} \to \{1,2\}$ in $\Fin$ 
with $f(1)=f(2)=1$,
$f_*: A \otimes_{\cA} A  \to A \otimes_{\cA} A$
is defined to be $m \otimes u$.
\end{enumerate}
\end{eg}

The case $\Fine$ is quite similar, and we omit the detail.

\begin{rmk}
The axiom of ring object 
implies that for commutative diagrams
\begin{align*}
\xymatrix{
 \{1,2,3\} \ar[r]^{f} \ar[d]_{g}
&\{1,2\}   \ar[d]^{h}
&
&\{1\}     \ar[r]^{i} \ar[d]_{j}
&\{1,2\}   \ar[d]^{h} 
\\
 \{1,2\}   \ar[r]_{h} 
&\{1\}
&
&\{1,2\}   \ar[r]_{h}
&\{1\} 
}
\end{align*}
in $\Fin$ with 
\begin{align*}
&f(1)=f(2)=1,\  f(3)=2, \quad
 g(1)=1,\  g(2)=g(3)=1, \quad 
 h(1)=h(2)=1,
\\
&i(1)=1,\quad
 j(1)=2,
\end{align*}
the diagrams 
\begin{align*}
\xymatrix{
 A^{\otimes 3} \ar[rr]^{f_* = m \otimes \id_A} 
               \ar[d]_{g_* = \id_A \otimes m}
&
&A^{\otimes 2} \ar[d]^{h_* = m}
&
&
&A             \ar[rr]^{i_* = (\id_A, u \circ t_A)} 
               \ar[d]_{j_* = (u \circ t_A, \id_A)}
&
&A^{\otimes 2} \ar[d]^{h_* = m} 
\\
 A^{\otimes 2} \ar[rr]_{h_* = m} 
&
&A
&
&
&A^{\otimes 2} \ar[rr]_{h_* = m}
&
&A 
}
\end{align*}
in $\cA$ commute.
\end{rmk}

One can check that 

\begin{lem}
\label{lem:cro}
Let $A$ be a commutative ring object in $\cA$.
Then the object $T_*(A)$ 
is a commutative ring object in $\Fun(\Fin,\cA)$.
\end{lem}

\begin{rmk}
If $A$ is not commutative, 
then $T_*(A)$ is not a ring object in $\Fun(\Fin,\cA)$.
Assume $A$ is a ring object in $\cA$
If  $T_*(A)$ is a ring objet in $\Fun(\Fin,\cA)$,
then there is a morphism $m_*: T_*(A) \otimes T_*(A) \to T_*(A)$ 
giving a multiplicative structure on  $T_*(A)$.
It means that for any morphism $f: I \to J$ in $\Fin$
we have a commuting diagram
$$
 \xymatrix{
      T_*(A)(I) \otimes_{\cA} T_*(A)(I) \ar[rr]^(0.6){m_*(I)} \ar[d]_{T_*(A)(f) \otimes T_*(A)(f)}
  & & T_*(A)(I) \ar[d]^{T_*(A)(f)}
  \\
      T_*(A)(J) \otimes_{\cA} T_*(A)(J) \ar[rr]_(0.6){m_*(J)} 
  & & T_*(A)(J)
 }
$$ 
in $\cA$.
Consider, for example, the morphism $f:\{1,2\} \to \{1\}$.
Then the above diagram becomes
$$
 \xymatrix{
      A^{\otimes 2} \otimes A^{\otimes 2} 
      \ar[rr]^{m_*(\{1,2\})} \ar[d]_{f_* \otimes f_* = m \otimes m}
  & & A^{\otimes 2} \ar[d]^{f_* = m}
  \\
      A \otimes A \ar[rr]_{m_*(\{1\})} 
  & & A
 }
$$ 
Unless $A$ is commutative, there is no canonical way of defining $m_*$ 
such that the above diagram commutes.
\end{rmk}

In a dual way, one can consider 

\begin{dfn}
Let $\cC$ be the category $\Fin$ or $\Fine$.
For a coalgebra object $C$ in $\cA$
we define an object $T^*(C)$ in $\Fun(\cC^{op},\cA)$ 
by $T^*(C)(I) := \otimes_{i \in I} C$ for $I \in \Ob(\cC)$,
and $T^*(C)(f): T^*(C)(J)  \to T^*(C)(I)$ 
for a morphism $f: I \to J$ in $\cC$ 
in a natural way by the comultiplication and the counit of $H$. 
We sometimes use the symbol $ f^* := T^*(C)(f)$ for simplicity.
\end{dfn}

\begin{lem}
For a cocommutative coalgebra object $C$ in $\cA$,
$T^*(C)$ 
is a cocommutative coalgebra object in $\Fun(\Fin,\cA)$.
\end{lem}

As for a bialgebra object, we have 

\begin{rmk}
For a bialgebra object $H$ in $\cA$,
we can consider $T_*(H)$ in $\Fun(\cC,\cA)$ 
and $T^*(H)$ in $\Fun(\cC^{op},\cA)$
using the algebra and coalgebra structure on $H$. 
By the axiom of bialgebra object,
for a commutative diagram
\begin{align*}
\xymatrix{
 \{1,2\} \ar[r]^{f} 
&\{1\}   
&\{1,2\} \ar[l]_{f} \ar@{=}[d]
\\
 \{1,2,3,4\} \ar[r]_{h} \ar[u]_{g}
&\{1,2,3,4\} \ar[r]_{g}
&\{1,2\}   
}
\end{align*}
in $\Fin$ with 
\begin{align*}
&f(1)=f(2)=1, \quad
 g(1)=g(2)=1,\  g(3)=g(4)=2, \quad 
 h(1)=1,\ h(2)=3,\ h(3)=2,\ h(4)=4,
\end{align*}
we have a commutative diagram
\begin{align}
\label{diag:bialg:D(ab)=D(a)D(b)}
\xymatrix{
 H^{\otimes 2} \ar[rrr]^{f_* = m } 
               \ar[d]_{g^* = \Delta \otimes \Delta}
&
&
&H
&             
&H^{\otimes 2} \ar[ll]_{f^* = \Delta} \ar@{=}[d]
\\
 H^{\otimes 4} \ar[rrr]_{h_* = \id_H \otimes \sigma_{H,H} \otimes \id_H}
&
&
&H^{\otimes 4} \ar[rr]_{g_* = m \otimes m}
&
&H^{\otimes 2}
}
\end{align}
\end{rmk}

One can introduce a module on a ring object in $\Fun(\cC,\cA)$,
although we don't write it down.
We will focus on modules of bialgebra objects 
in the following sense:

\begin{dfn}
\label{dfn:T-module}
Let $\cC$ be $\Fin$ or $\Fine$,
and $H$ be a bialgebra object in $\cA$.
Define a $T^*(H)$-module in $\Fun(\cC,\cA)$
to be an object $M$ of $\Fun(\cC,\cA)$ 
such that $M(I)$ is a module of
the ring object $T^*(H)(I)=\otimes_{i \in I}H$ 
(with component-wise multiplication)
for any $I \in \Ob(\cC)$ and 
such that the diagram
\begin{align}
\label{diag:H-mod}
\xymatrix{
 T^*(H)(I) \otimes M(I) \ar[rd]_{a(I)}
&
&T^*(H)(J) \otimes M(I) 
 \ar[ll]_{T^*(H)(f) \otimes \id_{M(I)}} 
 \ar[rr]^{\id_{M(J)} \otimes M(f)}
&
&T^*(H)(J) \otimes M(J) \ar[ld]^{a(J)}
\\
&
 M(I) \ar[rr]_{M(f)}
&
&M(J) 
}
\end{align}
in the category $\cA$
commutes for any morphism $f:I \to J$ in $\cC$.
Here the arrows $a(I)$ and $a(J)$ indicate the $H$-action on $M$,
and the $H$-action on the tensor product of modules 
is given by the comultiplication $\Delta$ of $H$ 
as usual.
\end{dfn}

If $\cA=\RM$, then the commutativity of the diagram \eqref{diag:H-mod} 
can be written as 
\begin{align}
\label{eq:H-mod}
 f_*(f^*(g).m)=g.f_*(m)
\end{align}
for any $g\in T^*(H)(J)$ 
and $m \in M(I)$, where we denoted by $.$ the $M$-action and 
$f_* =M(f)$.

\begin{eg}
\label{eg:bialgebra:module}
\begin{enumerate}
\item
For a bialgebra object $H$ in $\cA$,
the object $T_*(H)$ in $\Fun(\cC,\cA)$
is an $H$-module,
since $T_*(H)(I)=\otimes_{i \in I}H$ 
is a module of $T^*(H)(I)$ by the component-wise product, 
and since the commutativity of the diagram \eqref{diag:H-mod} 
can be checked by the bialgebra axiom.
For example, the case $f:\{1,2\} \to \{1\}$ 
follows from \eqref{diag:bialg:D(ab)=D(a)D(b)}.

\item
If $M$ is a ring object in $\cA$ 
with an action of a bialgebra object $H$,
then $T_*(M)$ is a $T^*(H)$-module in $\Fun(\cC,\cA)$.
\end{enumerate}
\end{eg}

One can check that 
$H$-modules in $\Fun(\cC,\cA)$ 
form an additive monoidal category.
Let us introduce

\begin{dfn}
Define $\Fun(\cC,\cA,T^*(H))$ 
to be the additive monoidal category 
of $T^*(H)$-modules in $\Fun(\cC,\cA)$.
\end{dfn}

If $H$ is cocommutative,
then $\Fun(\cC,\cA,T^*(H))$ becomes a symmetric monoidal category.

\begin{eg}
If $M$ is a commutative ring object in $\cA$ 
with action of a cocommutative bialgebra object $H$,
then $T_*(M)$ is a commutative ring object in $\Fun(\Fin,\cA,T^*(H))$.
\end{eg}

Remarking that one can define 
the category of modules over a commutative ring object
in an additive symmetric monoidal category,
and that it is again an additive symmetric monoidal category,
we introduce 

\begin{dfn}
Let $H$ be a cocommutative bialgebra object in $\cA$
and let $S$ be a commutative ring object in $\Fun(\cC,\cA,T^*(H))$.
Define $\Fun(\cC,\cA,T^*(H),S)$ to be 
the additive symmetric monoidal category of $S$-modules.
\end{dfn}

The letter $S$ means `singular',
and the object $S$ encodes the singular parts 
of OPEs of the fields considered.
The $(\cA,H,S)$-vertex algebra is defined to be 
a \emph{singular} commutative ring object in $\Fun(\Fin,T^*(H),S)$.
The term \emph{singular} is clarified by the following notion.

\begin{dfn}
Let $\cC$ be $\Fin$ or $\Fine$.
Let $H$ be a cocommutative bialgebra object in $\cA$
and let $S$ be a commutative ring object in $\Fun(\cC,\cA,T^*(H))$.
For objects $U_1,U_2,\ldots,U_n$ and $V$ of $\Fun(\cC,\cA,T^*(H),S)$,
define the singular multilinear map 
to be a family of maps 
$$
 U_1(I_1) \otimes_{\cA} U_2(I_2) \otimes_{\cA} \cdots
 \otimes_{\cA} U_n(I_n)
 \longto
 V(I_1 \sqcup I_2 \sqcup \cdots \sqcup I_n)
$$
for any $I_1,I_2,\ldots,I_n \in \Ob(\cC)$
satisfying the following conditions.
\begin{enumerate}
 \item
 The maps commute with the action of $T^*(H)$.
 \item
 The maps commute with the actions of $S(I_1),S(I_2),\ldots,S(I_n)$.
 \item
 For morphisms $I_1 \to I'_1$, $I_2 \to I'_2$, $\ldots$, $I_n \to I'_n$
 in $\cC$, the diagram 
 \begin{align*}
  \xymatrix{
   U_1(I_1) \otimes U_2(I_2) \otimes \cdots \otimes U_n(I_n)
   \ar[r] \ar[d]
   & V(I_1 \sqcup I_2 \sqcup \cdots \sqcup I_n) \ar[d]
   \\   
   U_1(I'_1) \otimes U_2(I'_2) \otimes \cdots \otimes U_n(I'_n)
   \ar[r]
   & V(I'_1 \sqcup I'_2 \sqcup \cdots \sqcup I'_n)
  }
 \end{align*}
 in $\cA$ commutes.
\end{enumerate}
\end{dfn}

Since we assumed 
that $\cA$ is cocomplete and colimits commute with tensor products,
the singular multilinear maps are representable.
Thus the following definition makes sense.

\begin{dfn}
For objects $U_1,U_2,\ldots,U_n$ of $\Fun(\cC,\cA,T^*(H),S)$,
the singular tensor product 
$U_1 \odot U_2 \odot \cdots \odot U_n$ 
is the object 
in $\Fun(\cC,\cA,T^*(H),S)$
representing the singular multilinear maps
$U_1(I_1) \otimes \cdots  \otimes  U_n(I_n) \to
 V(I_1 \sqcup \cdots \sqcup I_n)$.
\end{dfn}

The singular tensor product can be expressed explicitly as  
$$
 (U_1 \odot U_2 \odot \cdots \odot U_n)(I)
 := \varinjlim_{ \bigsqcup_{i=1}^n I_i\to I }
 (U_1(I_1) \otimes U_2(I_2)\otimes \cdots  \otimes  U_n(I_n))
 \bigotimes_{S(I_1)\otimes S(I_2) \otimes \cdots \otimes S(I_n)}
 S(I),
$$
where the colimit is taken over the following category.
An object $\bigsqcup_{i=1}^n I_i\to I$
consists of $I_1,I_2,\ldots,I_n \in \Ob(\cC)$
with a morphism from 
$I_1 \sqcup I_2 \sqcup \cdots \sqcup I_n $ to $I$
in $\cC$,
and a morphism from $\bigsqcup_{i=1}^n I_i\to I$ 
to $\bigsqcup_{i=1}^n I'_i\to I$ 
consists of morphisms $I_i \to I'_i$ ($i=1,2,\ldots,n$)
making the diagram
$$
\xymatrix@R=3ex{
 I_1 \sqcup I_2 \sqcup \cdots \sqcup I_n  \ar[r] \ar[d]
&I \ar@{=}[d]
\\
 I'_1 \sqcup I'_2 \sqcup \cdots \sqcup I'_n  \ar[r] 
&I
}
$$
in $\cC$ commutative.

One can check that the category appearing above
is a filtered (in the sense of \cite[Chap. IX]{M}) small category,
so that the colimit is in fact the filtered inductive limit
(or the direct limit).

For $\cC=\Fin$, the disjoint union $\sqcup$ is a coproduct,
which implies that the singular tensor product $\odot$ 
is the same as the ordinary tensor product $\otimes$.

By the definition of $\odot$,
there is a canonical morphism from $U_1 \odot U_2$ 
to $U_1 \otimes U_2$,
so that any ring object automatically 
has another ring structure with multiplication given by 
singular tensor products.
Thus the following definition makes sense.

\begin{dfn}
A singular ring object in $\Fun(\Fine,\cA,T^*H,S)$
is a ring object whose multiplicative structure 
is given by the singular tensor product $\odot$.
\end{dfn}

A ring object $S$ in $\Fun(\Fine,\cA,T^*(H))$ 
can be seen as a ring object in $\Fun(\Fin,\cA,T^*(H))$ 
by restriction under 
the embedding \eqref{eq:Fin:embed} of $\Fin$ into $\Fine$.
Then we can embed the category $\Fun(\Fin,\cA,T^*(H),S)$
into $\Fun(\Fine,\cA,T^*(H),S)$ by defining
\begin{align}
\label{eq:V:ext}
 V(I_1:I_2:\cdots:I_n)
 := V(I_1 \sqcup I_2 \sqcup \cdots \sqcup I_n)
    \bigotimes_{S(I_1) \otimes S(I_2) \otimes \cdots \otimes S(I_n)}
    S(I_1:I_2:\cdots:I_n)
\end{align}
for $V$ in $\Fun(\Fin,\cA,T^*(H),S)$.
Here $I_1:I_2:\cdots:I_n$ is an object of $\Fine$,
which is the disjoint union of $I_j$'s as a set,
and where the equivalence relation is defined 
so that each $I_j$ is the equivalence class.
For example, for $I_1=\{1\}$ and $I_2=\{1,2\}$,
we have $I_1:I_2=\{1:2,3\}$.

Thus the following definition makes sense.

\begin{dfn}
\label{dfn:scom}
A singular commutative ring object 
in $\Fun(\Fin,\cA,T^*H,S)$ 
is an object 
such that its extension \eqref{eq:V:ext} 
gives a singular commutative ring object 
in $\Fun(\Fine,\cA,T^*H,S)$.
\end{dfn}

Now we can introduce the main object.

\begin{dfn}
Let $\cA$ be an additive symmetric monoidal category,
$H$ be a cocommutative bialgebra object in $\cA$,
and $S$ be a commutative ring object in 
the additive symmetric monoidal category $\Fun(\Fine,\cA,T^*(H))$.
Define an $(\cA,H,S)$-vertex algebra 
to be a singular commutative ring in $\Fun(\Fin,\cA,T^*(H),S)$.
\end{dfn}

An $(\cA,H,S)$-vertex algebra $V$ is thus an object in 
$\Fun(\Fin,\cA)$,
although we often consider it as an object in 
$\Fun(\Fine,\cA)$ by the extension \eqref{eq:V:ext}.


\subsection{Relation to ordinary vertex algebras}
\label{subsec:RM}

Let $R$ be a commutative ring.
In the case $\cA=\RM$,
one can consider the following bialgebra.

\begin{dfn}
\label{dfn:FG}
Let $\FG$ be the commutative cocommutative bialgebra 
over $R$ with basis 
$\{ D^{(i)} \mid i \in \bb{Z}_{\ge0}\}$,
multiplication 
$D^{(i)}D^{(j)}=\binom{i+j}{i}D^{(i+j)}$
and comultiplication 
$\Delta(D^{(i)})=\sum_{j=0}^{i}D^{(i)} \otimes D^{(i-j)}$.
\end{dfn}

$\FG$ is the formal group ring of the one-dimensional 
additive formal group
(corresponding to the formal group law 
 $F(X,Y)= X+Y$).
Symbolically one has $D^{(i)}=D^i/i!$.

An important example for a commutative ring object $S$ 
in $\Fun(\Fine,\RM,T^*(\FG))$ is 

\begin{dfn}
\label{dfn:Se}
Define an object $\Se$ in $\Fun(\Fine,\RM)$ by 
\begin{align}
\label{eq:Se}
 \Se(I) := R[(x_i-x_j)^{\pm1} \mid i\not\equiv j \text{ in } I]
\end{align}
for $I \in \Ob(\Fine)$,
and 
$$
 \Se(f):\Se(I) \longto \Se(J), \quad 
 (x_i -x_j) \longmapsto (x_{f(i)}-x_{f(j)})
$$ 
for $f \in \Fine(I,J)$.
\end{dfn}

One can easily check that $\Se$ is indeed an object of $\Fun(\Fine,\RM)$.
One further has 

\begin{lem}
\label{lem:Se}
$\Se$ is a commutative ring object
in $\Fun(\Fine,\RM,T^*(\FG))$,
where the action of $\FG$ on $\Se$ is given by 
the derivation.
More explicitly, one has 
$D^{(i)}(x^m) = \binom{m}{i}x^{m-i}$.
\end{lem}


Let $V$ be an $(\RM,\FG,\Se)$-vertex algebra.
It is an object of $\Fun(\Fin,\RM,T^*(\FG),\Se)$,
so $V(I)$ is just an $R$-module for each $I \in \Ob(\Fin)$.
Let us look at the definition of singular tensor product 
for two $V$'s:
$$
 (V \odot V)(I)
 = \varinjlim_{\sqcup_{i=1}^2  I_i \to I}
  \bigl(V(I_1) \otimes V(I_2)\bigr) \bigotimes_{\Se(I_1) \otimes \Se(I_2)} \Se(I).
$$ 
Fix objects $I_1,I_2 \in \Ob(\Fin)$ and
take arbitrary elements $v_1 \in V(I_1)$ and $v_2 \in V(I_2)$.
The ordinary product $v_1 v_2$ is defined in $V(I_1 \sqcup I_2$).
By Definition \ref{dfn:scom} and the extension \eqref{eq:V:ext},
the singular tensor product $v_1 \odot v_2 $ 
is defined in $ (V \odot V)(I_1:I_2) \subset V(I_1:I_2)$
with $V(I_1:I_2)=V(I)\otimes_{\Se(I_1) \otimes \Se(I_2)}\Se(I_1:I_2)$,
and the singular commutativity of $V$ means 
$v_1 \odot v_2 = v_2 \odot v_1$ in $V(I_1:I_2)$.

In particular, setting $I_1=\{1\}$ and $I_2=\{2\}$, 
we have $I= I_1 \sqcup I_2 = \{1,2\}$ and $I_1:I_2=\{1:2\}$,
so that $S(I_1)=S(I_2)=R$ and $S(I_1:I_2)=R[(x_1-x_2)^{\pm1}]$,
hence we have $ V(I_1:I_2)=V(\{1,2\})[(x_1-x_2)^{\pm1}]$
and in this module the equation $v_1 \odot v_2 = v_2 \odot v_1$ holds.

Now we can recall the following main theorem in \cite{B:2001}:

\begin{fct}[{\cite[Theorem~4.3]{B:2001}}]
\label{fct:Borcherds}
Let $V$ be an $(\RM,\FG,\Se)$-vertex algebra.
Then  $V(\{1\})$ has a structure of ordinary vertex algebra 
over the ring $R$.
\end{fct}

For the sake of completeness, 
let us write down the axiom of ordinary vertex algebra here.

\begin{dfn}\label{dfn:ova}
An ordinary vertex algebra defined over a commutative ring $R$
is a colletion of data 
\begin{itemize}
\item
(space of fields)
an $R$-module $V$
\item
(vacuum)
an element $\vac \in V$ 
\item
(translation)
an $R$-linear operator $T: V \to V$ 
\item
(vertex operators)
an $R$-linear operation 
$Y(\;,z): V \to \End\bigl( V \bigr)[[z^{\pm1}]]$ 
\end{itemize}
satisfying the following axioms.
\begin{itemize}
\item
(vacuum axiom)
$Y(\vac,z)=\id_V$ and 
$Y(A,z)\vac \in A + z R[[z]]$ for any $A \in V$.
\item
(translation axiom)
$[T,Y(A,z)]=\partial_z Y(A,z)$ for any $A \in V$
\item
(locality axiom)
$\{Y(A,z)\mid A \in V\}$ are mutually local,
that is, 
for any $A,B \in V$ there exists $N \in \bb{Z}_{\ge0}$ such that 
$(z-w)^n[Y(a,z),Y(b,w)]=0$ as a formal power series in 
$\End\bigl( V \bigr) [[z^{\pm1},w^{\pm1}]]$.
\end{itemize}
\end{dfn}

Let us sketch the proof of Fact~\ref{fct:Borcherds} briefly.
For details see \cite[Proof of Theorem 4.3]{B:2001}.
\cite[(4.3) Proof of Theorem 1]{P} also gives a nice demonstration.

\begin{proof}[Proof of Factt~\ref{fct:Borcherds}]
The vacuum $\vac$ is defined to be $1$ in the $R$-algebra $V(\{1\})$.

The translation $T$ is defined by the action of $\FG$ on $V(\{1\})$.
In other words, $T := D^{(1)}$.

We want to make an $R$-linear map 
$$
 Y(\;,x_1): V(\{1\}) \to \End_R\bigl( V(\{1\}) \bigr) [[x_1]][x_1^{-1}]
$$ 
satisfying the axiom of ordinary vertex algebra.
For $u_1,u_2 \in V(\{1\})$,
we have $u_1 \odot u_2 = u_2 \odot u_1$ in 
$V(\{1:2\})=V(\{1,2\})[(x_1-x_2)^{\pm1}]$
as remarked in the paragraph before Fact \ref{fct:Borcherds}.
Recalling the action of $\FG$ on $\Se$,
we may consider the ``Taylor series expansion''
\begin{align}
\label{eq:taylor}
 V(\{1,2\}) \longto V(\{1\})[[x_1,x_2]],
 \quad
 w \longmapsto \sum_{i,j\ge0} f_*( D_1^{(i)}D_2^{(j)} w) x_1^i x_2^j,
\end{align}
where $f: \{1,2\} \to \{1\}$ is a morphism in $\Fin$ and 
$D_1,D_2$ indicate the two different actions of $\FG$ on $V(\{1,2\})$.
Combining this expansion with the extension \eqref{eq:V:ext},
we have an $R$-linear map from $V(\{1,2\})[(x_1 - x_2)^{\pm 1}]$ 
to $V(\{1\})[[x_1,x_2]][(x_1-x_2)^{-1}]$,
and we denote the image of $u_1 \odot u_2$ under this map 
by $u_1(x_1)u_2(x_2)$.
Then define $Y(v_1,x_1)$ by 
$$
 u_2 \longmapsto u_1(x_1)u_2(0) \in  
 \left. 
  V(\{1\})[[x_1,x_2]][(x_1-x_2)^{-1}]
 \right|_{x_2=0} = 
 V(\{1\})[[x_1]][x_1^{-1}].
$$

As for the check of vertex algebra axioms,
the most non-trivial part is the locality axiom,
which is a consequence of the singular commutativity 
$u_1 \odot u_2 = u_2 \odot u_1$.
Indeed, the singular commutativity 
implies 
$$
 (x_1-x_2)^N (u_1(x)u_2(x)-u_2(x)u_1(x))u_3 = 0
$$
with some $N$, which depends only on $u_1$ and $u_2$.
This is nothing but the locality.

The translation axiom comes from the action of $\FG$.
The vacuum axiom is the consequence of 
the singular commutativity with respect to 
$A \in V(\{1\})$ and $1 \in V(\{1\})$.
We omit the detailed discussion.
\end{proof}

\begin{dfn}
For an  $\AHS$-vertex algebra $V$,
$V(\{1\})$ is called 
the ordinary vertex algebra associated to $V$.
\end{dfn}

$\AHS$-vertex algebras form an abelian category.
Moreover they form a symmetric monoidal category under 
the tensor product $\otimes$.
These structures induces the same ones on 
the ordinary vertex algebras,
which are described in \cite[\S1.3]{FB} for example.

\begin{rmk}
As mentioned in \cite[Example 4.9]{B:2001},
$\Fun(\Fin,\cA,T^*(H),S)$ is not closed 
under the singular tensor product $\odot$,
so that one should consider $\otimes$ for the monoidal structure 
on $\AHS$-vertex algebras.
\end{rmk}

\subsection{Twisting construction}
\label{subsec:twist}

In the next section
we will reconstruct several $\AHS$-vertex algebras,
to which the associated ordinary vertex algebras 
are well-known ones:
Heisenberg algebras, affine Kac-Moody Lie algebras,
the lattice vertex algebras and so on.
For this purpose,
let us recall the twisted group construction of 
$(\cA,H,S)$-vertex algebra,
which was introduced in \cite{B:2001} and 
investigated in detail in \cite{P}.

\begin{dfn}
Let $R$ be a commutative ring, 
$M$ and $N$ be bialgebras over $R$, 
and $S$ be a commutative algebra over $R$.

\begin{enumerate}
\item
A bimultiplicative map from $M \otimes_R N$ to $S$ 
is an $R$-linear map $r: M \otimes_R N \to S$ such that 
\begin{align*}
 r(a\otimes 1) &= \ve_M(a), \qquad
 r(1\otimes a)  = \ve_N(a),
\\
 r(ab \otimes c) &=\sum r(a\otimes c')r(b\otimes c''),
\\
 r(a \otimes b c) &=\sum r(a' \otimes b)r(a'' \otimes c)
\end{align*}
hold for any $a,b,c \in M$.

\item
A bimultiplicative map on $M\otimes_R M$ to $S$ 
is called an $S$-valued bicharacter.

\item
A bicharacter $r$ is called symmetric if 
$$
 r(a \otimes b )= r(b \otimes a)
$$ 
holds for any $a, b \in M$.
\end{enumerate}
\end{dfn}

The following lemma is due to \cite[Lemma/Definition 2.6]{B:2001},
where $M$ is assumed to be commutative.

\begin{lem}
Suppose $r$ is an $S$-valued bicharacter 
of a cocommutative bialgebra $M$ over $R$.
\begin{enumerate}
\item
The operation
\begin{align}
\label{eq:twisted_product}
 a \circ_r b := \sum a' b' r(a''\otimes b'')
\end{align}
defines a unital associative algebra $(M \otimes_R S,\circ_r,1_M)$ over $R$,
where $1_M$ is the unit of the original algebra structure on $M$.
\item
If $r$ is symmetric and $M$ is commutative,
then the new algebra $(M \otimes_R S,\circ_r,1_M)$ 
is commutative.
\end{enumerate}
\end{lem}

\begin{proof}
We only indicate the proof of the associativity.
On one side we have
\begin{align*}
(a \circ_r b) \circ_r c
&=\bigl( \sum a' b' r(a'' \otimes b'') \bigr) \circ_r c
\\
&=\sum (a' b')' c' r((a'b')''\otimes c'') r(a'' \otimes b'') 
\\
&=\sum (a' b') c' r(a'' b'' \otimes c'') r(a''' \otimes b''') 
\\
&=\sum (a' b') c' r(a'' \otimes c'') r(b'' \otimes c''') r(a''' \otimes b'''), 
\end{align*}
where in the third line we used the notation
$\bigl((\Delta \otimes 1)\circ\Delta\bigr) (a) 
= \sum a' \otimes a'' \otimes a'''$.
On the other side we have
\begin{align*}
a \circ_r (b \circ_r c) 
&=a \circ_r \bigl( \sum b' c' r(b'' \otimes c'') \bigr) 
\\
&=\sum a'(b' c')' r(a'' \otimes (b'c')'') r(b'' \otimes c'') 
\\
&=\sum a'(b' c') r(a'' \otimes (b''c'') r(b''' \otimes c''')
\\
&=\sum a'(b' c') r(a'' \otimes b'') r(a''' \otimes c'') r(b''' \otimes c''').
\end{align*}
Since $M$ is an associative algebra, we have
$$
(a' b') c' = a'(b' c').
$$
Since $M$ is a cocommutative coalgebra, we have
$$
\sum (a' b') c' r(a'' \otimes c'') r(b'' \otimes c''') r(a''' \otimes b''')
=
\sum (a' b') c' r(a'' \otimes b'') r(a''' \otimes c'') r(b''' \otimes c''').
$$
Therefore we have the conclusion.
\end{proof}

\begin{dfn}
The algebra $(M \otimes_R S,\circ_r,1_M)$ constructed in the previous Lemma
is called the twisting of $M$ by $r$
and denoted by $\wt{M}$ or $M^r$.
\end{dfn}

As in \S\ref{subsec:RM}, 
we will consider the case 
where $M$ has an action
of a cocommutative coalgebra (or bialgebra) $H$.
There is a a universal ring with $H$-action in the following sense.

\begin{fct}[{\cite[Lemma/Definition 2.10]{B:2001}}]
Suppose $M$ is an $R$-algebra and $H$ is an $R$-coalgebra.
Then there is a universal $R$-algebra $H(M)$ such that 
there is a map 
$$
 H \otimes M \longto H(M),\quad
 h \otimes m \mapsto h(m)
$$
with
$$
 h(mn)= \sum h'(m)h''(n),\quad
 h(1)=\ve_H(h).
$$
If $M$ is commutative and $H$ is cocommutative,
then $H(M)$ is commutative.
If $H$ is a bialgebra, then $H$ acts on the algebra $H(M)$.
If $M$ is a bialgebra, then $H(M)$ is also a bialgebra.
\end{fct}
\begin{proof}
$H(M)$ is defined to be the quotient of the tensor algebra of $H\otimes M$ 
by the ideal generated by the desired relations.
The rest statements are easy to check.
\end{proof}

It is natural to introduce 

\begin{dfn}
\label{dfn:inv-bichar}
Let $M$ be an $R$-bialgebra and
$S$ be a commutative $R$-algebra.
Suppose that an $R$-coalgebra $H$ acts on $M$ and 
$H \otimes H$ acts on $S$.
An $S$-valued bicharacter $r$ on $M$ is called $H$-invariant if 
$$
 r\bigl( (g a) \otimes (h b) \bigr)  
 = (g \otimes h) \bigl( r( a\otimes b) \bigr) 
$$
holds for any $g,h \in H$ and $a,b \in M$.
\end{dfn} 

Then we also have 

\begin{fct}[{\cite[Lemma 2.15]{B:2001}}]
\label{fct:bichar-HM}
Let $H$ be a cocommutative bialgebra, 
$S$ be a commutative algebra acted on by $H \otimes H$,
and $M$ be a commutative cocommutative bialgebra 
with an $S$-valued bicharacter $r$
Then $r$ extends uniquely to an $H$-invariant 
$S$-valued bicharacter on $H(M)$.
\end{fct}

The discussion above can be generalized 
to the categorical setting given in \S \ref{subsec:cat}.
For example, one can define an $S$-valued bicharacter on $M \otimes_{\cA} M$,
where $M$ is a bialgebra object in the additive monoidal category $\cA$
and $S$ is a commutative ring object in $\cA$.,

Now we  recall the construction of $(\cA,H,S)$-vertex algebra 
using bicharacter, which was explained in \cite[Lemma 4.1, Theorem 4.2]{B:2001}.

\begin{lem}[{\cite[Lemma 4.1]{B:2001}}]
\label{lem:sing-bichar}
Let $M$ be a commutative and cocommutative bialgebra object in $\cA$,
$H$ be a cocommutative bialgebra object $\cA$,
and $S$ be a commutative ring object in $\Fun(\Fine,T^*(H))$
If $r$ is an $S(\{1:2\})$-valued $H$-invariant bicharacter 
on a commutative cocommutative bialgebra $H(M)$ in $\cA$,
then one can extend $r$ to a singular bicharacter of $T_*(H(M))$.
\end{lem}

Note that  $S(\{1:2\})$ has an $H \otimes H$-action 
since $S$ is a $T^*(H)$-module so that 
$S(\{1:2\})$ is a module over $T^*(H)(\{1:2\}) = H \otimes H$ 
by Definition \ref{dfn:T-module}.
So the term `$H$-invariant' makes sense by Definition \ref{dfn:inv-bichar}.

Let us briefly sketch the proof of Lemma~\ref{lem:sing-bichar}.
We define the extended $r$ on $T_*(H(M))(I \sqcup J)$ by
$$
 r \Bigl( \bigotimes_{i \in I} a_i \otimes \bigotimes_{j \in J} b_j \Bigr)
 := 
 \sum \prod_{i \in I} \prod _{j \in J} r(a_i^{(j)} \otimes b_j^{(i)})
$$
with 
$\Delta_M^{|J|-1}(a_i) = \sum \bigotimes_{j \in J} a_i^{(j)}$ and
$\Delta_M^{|I|-1}(b_j) = \sum \bigotimes_{i \in I} b_j^{(i)}$.
Each $r(a_i^{(j)} \otimes b_j^{(i)})$ is considered 
as an element of $S(I \sqcup J)$ using the natural map 
from $S(\{i:j\})$ to $S(I \sqcup J)$.

\begin{fct}[{\cite[Theorem 4.2]{B:2001}}]
\label{fct:twisting}
Suppose that $H$ is a cocommutative bialgebra in $\cA$ 
and that $S$ is a commutative ring in $\Fun(\Fine,\cA,T^*(H))$.
Assume that $r$ is a symmetric $S(\{1:2\})$-valued bicharacter 
of a commutative and cocommutative bialgebra $M$ in $\cA$.
Then the twisting $T_*(H(M))^r$ of $T_*(H(M))$ by the singular bicharacter 
constructed by the extension of $r$ in Lemma \ref{lem:sing-bichar}
is $(\cA,H,S)$-vertex algebra.
\end{fct}




The following remark due to \cite[(4.2)]{P} is useful.

\begin{lem}
Let $T_*(\FG(M))$ be an $\RHS$-vertex algebra 
resulting from a universal commutative cocommutative bialgebra $\FG(M)$.
Consider the twisting $T_*(\FG(M))^r$ of $T_*(\FG(M))$ 
by the (singular) bicharacter $r$.
Then in the ordinary vertex algebra associated to $T_*(\FG(M))^r$
we have 
$$
 Y(a,x_1)Y(b,x_2)\vac = \Phi_r(a,b) \in R[(x_1-x_2)^{\pm1}]
$$
for $a,b \in T_*(\FG(M))^r(\{1\}) = \FG(M)$ with
$$
 \Phi_r(a,b) := \sum_{i,j\ge 0, \, (a), (b)} x_1^i x_2^j D^{(i)}(a')D^{(j)}(b'')r(a'',b'')
$$
\end{lem}
\begin{proof}
This is the direct consequence of the formula \eqref{eq:taylor} 
and the definition of the twisted product \eqref{eq:twisted_product}.
\end{proof}

\section{Examples of $\AHS$-vertex algebras}

\subsection{Heisenberg algebra}
\label{subsec:Heisenberg}

We introduce a typical example of $\AHS$-vertex algebra,
whose ordinary vertex algebra will be the Heisenberg vertex algebra.
Let $R$ be a fixed commutative ring, and 
let us set $\cA = \RM$, $H = \FG$ and $S = \Se$ 
as in \S\ref{subsec:RM}.

Consider the Laurent polynomial ring $R[t^{\pm1}]$ of one variable.
It is (trivially) a commutative ring object in $\RM$,
and has an action of $\FG = R[D^{(i)} \mid i \in\bb{Z}_{\ge0}]$ 
defined as
$$
 D^{(i)}t^n = \binom{n}{i}t^{n-i}.
$$ 
By Lemma \ref{lem:cro}, the object $T_*(R[t^{\pm1}])$ in $\Fun(\cC,\RM)$ 
is a commutative ring object.
Hereafter let us use the notation
$T_*(R[t^{\pm1}])(\{1,\ldots,n\}) = R[t_1^{\pm1},\ldots,t_n^{\pm1}]$.

We also have
\begin{lem}
$T_*(R[t^{\pm1}])$ is a $T^*(\FG)$-module.
\end{lem}
\begin{proof}
We only need to check the formula \eqref{eq:H-mod} 
with $f$ given by each case in Example \ref{eg:f}.
The case $f=\id$ is trivial.
In the case $f:\{1,2\} \to \{1\}$,
we may set $g= D^{(k)}$ and $m= t_1^m \otimes t_2^n$.
Then
\begin{align*}
f_*(f^*(g).m) 
 &= f_*\Bigl(\sum_{i+j=k}
              (D^{(i)}\otimes D^{(j)})(t_1^m \otimes t_2^n)\Bigr)
\\
 &= f_*\Bigl(\sum_{i+j=k}\binom{m}{i}\binom{n}{j}
                         t_1^{m-i} \otimes t_2^{n-j})\Bigr)
\\
 &= \sum_{i+j=k}\binom{m}{i}\binom{n}{j}t_1^{m+n-k}
 = \binom{m+n}{k}t_1^{m+n-k} = g.f_*(m).
\end{align*}
In the second from last equality is the result of the binomial formula.
The case $f:\{1,\ldots,n\} \to \{1\}$ can be treated similarly.
The case $f:\emptyset \to \{1\}$ is 
the result of $D^{(i)} 1 = 0$ for $i>0$.
The last case $f=s:\{1,2\} \to \{1,2\}$ 
is the consequence of the commutativity of the polynomial ring.
\end{proof}

$T_*(R[t^{\pm1}])$ is obviously an $\Se$-module 
(with the action given by the multiplications of rational functions),
and $T_*(R[t^{\pm1}])$ is a commutative ring object 
in $\Fun(\cC,\RM,T^*(\FG),\Se)$.
We also have the commutative subalgebra  $T_*(R[t^{-1}])$ 
in $\Fun(\cC,\RM,T^*(\FG),\Se)$. 

Next we consider the one-dimensional 
commutative Lie algebra $\frk{c} = R b$.
Its universal enveloping algebra $U(\frk{c})$ 
is isomorphic to the polynomial ring $R[b]$, 
or the symmetric algebra $S^{\bullet}b$.
$U(\frk{c})$ has a cocommutative bialgebra structure 
with $\Delta(b)=b\otimes 1 + 1 \otimes b$.

The tensor product
$L\frk{c} \equiv \frk{c}[t^{\pm1}] := \frk{c} \otimes_R R[t^{\pm1}]$
also has the structure of commutative Lie algebra,
which may be called the loop Lie algebra (attached to $\frk{c}$).
It has a Lie subalgebra 
$\frk{c}[t] := \frk{c} \otimes_R R[t]$.
One may consider the universal enveloping algebras 
$U(L\frk{c}) \supset U(\frk{c}[t])$.
The trivial representation $R_0 := R v_0$ of $\frk{c}[t]$,
where $v_0$ is a basis of the representation space,
induces the Verma module 
$\pi_0 := \Ind_{\frk{c}[t]}^{L\frk{c}}$ of $U(L\frk{c})$.

Hereafter we use the notation
$$
 b_i := b \otimes t^i
$$ 
for $i \in \bb{Z}$.
$\pi_0$ has a basis consisting of the monomials 
$$
  b_{-i_1}b_{-i_2} \cdots b_{-i_n}v_0,\quad
  n\ge0,\ i_1 \ge i_2 \ge \cdots \ge i_n>0.
$$
Since $\frk{c}$ is commutative, 
one has a unital associative commutative ring structure on $\pi_0$ 
defined by 
$$ 
 b_{-i} v_0 \otimes b_{-j} v_0 \mapsto  b_{-i} b_{-j} v_0 .
$$

Then, as in the case of $R[t^{-1}]$, 
Lemma \ref{lem:cro} says that the object $T_*(\pi_0)$ 
in $\Fun(\cC,\RM)$ is a commutative ring object.
The $T^*(\FG)$-module structure on $T_*(R[t^{-1}])$ 
induces one on $T_*(\pi_0)$.
Thus $T_*(\pi_0)$ is an object of $\Fun(\cC,\RM,T^*(\FG))$.
Similarly one can see that $T_*(\pi_0)$ is an $\Se$-module,
and it is also a (singular) commutative ring object 
in $\Fun(\cC,\RM,T^*(\FG),\Se)$.
Therefore we have an $\AHS$-vertex algebra $T_*(\pi_0)$.

Let us describe the ordinary vertex algebra associated to $T_*(\pi_0)$.
We use the notations
$$
 T_*(\pi_0)(\{1\}) = R[b_{-i} \mid i \ge 0] v_0
                   = R [b_{-i}^{(1)} \mid i \ge 0] v_0
$$  
and 
$$
 T_*(\pi_0)(\{1,\ldots,n\})
 = R[b_{-i}^{(j)} \mid i\ge0, \, n\ge j\ge1 ] v0.
$$ 
Recalling the proof of Fact \ref{fct:Borcherds},
we compute several vertex operators $Y(\;,z)$.

\begin{align*}
Y(b_{-n}v_0,x_1)b_{-k}v_0
&=\sum_{i,j\ge0}x_1^ix_2^jf_{*}\Bigl(
    (D^{(i)} \otimes D^{(j)})(b_{-n}^{(1)}v_0\otimes b_{-k}^{(2)}v_0)
  \Bigr)
  \Bigr|_{x_2=0}
\\
&=\sum_{i,j\ge0}x_1^ix_2^j \binom{n+i-1}{i}\binom{k+j-1}{j}
   f_{*}\bigl( b_{-n-i}^{(1)}v_0\otimes b_{-k-j}^{(2)}v_0 \bigr)
  \Bigr|_{x_2=0}
\\
&=\sum_{i,j\ge0}x_1^ix_2^j \binom{n+i-1}{i}\binom{k+j-1}{j}
    b_{-n-i}^{(1)} b_{-k-j}^{(1)} v_0
  \Bigr|_{x_2=0}
\\
&=\sum_{i\ge0}x_1^i \binom{n+i-1}{i} 
   b_{-n-i} b_{-k} v_0
\end{align*}
Here we used the notation $f:\{1,2\} \to \{1\}$, a morphism in $\Fin$.

A similar calculation gives
\begin{align*}
&Y(b_{-m_1} \cdots b_{-m_k}v_0,z)b_{-n_1}\cdots b_{-n_l}v_0
\\
&=\sum_{i,j_1,\ldots,j_k\ge0} x_1^i 
   \binom{i}{j_1,\ldots,j_k}
   \binom{m_1+j_1-1}{j_1}
   \cdots
   \binom{m_k+j_k-1}{j_k}
   b_{-m_1-j_1} \cdots b_{-m_k-j_k} b_{-n_1} \cdots b_{-n_l} v_0
\end{align*}

These formulas correspond to the OPE
$$
 \partial_z^i b(z) \partial_w^j b(w) 
 =\nord{ \partial_z^i b(z) \partial_w^j b(w)}
$$
with $b(z) := \sum_{n\in \bb{Z}} b_n z^{-n-1}$,
where $\nord{}$ is the usual normal ordering.
The Goddard's uniqueness Theorem \cite[3.1.1]{FB}
and the reconstruction theorem \cite[2.3.11]{FB} 
imply that the ordinary vertex algebra attached to $\pi_0$ 
coincides with the Heisenberg vertex algebra without central extension.
In order to construct the usual Heisenberg algebra with central extension,
where the OPE reads
$$
 b(z)b(w) = \dfrac{k}{(z-w)^2} + \nord{b(z)b(w)}
$$
we need to recall the twisting construction 
reviwed in \S\ref{subsec:twist}.

\begin{lem}
$\pi_0$ coincides with the universal algebra $H_a(U(\frk{c}))$.
\end{lem}
\begin{proof}
$H_a(U(\frk{c}))$ is isomorphic to $R[D^{(i)}] \otimes_R R[b]$,
so it is isomorphic to $\pi_0$ under the map 
$$ 
 D^{(i)}b^n \longmapsto  
  \sum_{\substack{i_1,\ldots,i_n\ge0 \\ i_1 + \cdots + i_n = i}}
  b_{-i_1-1} \cdots b_{-i_n-1}v_0.
$$
\end{proof}

Since $H_a$ is cocommutative and $U(\frk{c})$ is a cocommutative commutative bialgebra,
we may apply the twisting construction.
Let us consider the following bicharacter.
Fix an element $c$ in $R$.
The $R$-bilinear map
$b \otimes b \mapsto c$ on $R b$
induces an $R$-valued bicharacter  on $U(\frk{c})$
given by $b^m \otimes  b^n \mapsto m! c^m \delta_{m,n}$.
Now consider the $\Se(\{1:2\})$-valued bicharacter 
$$ 
 r(b^m \otimes  b^n) = \dfrac{m! c^m \delta_{m,n}}{(x_1-x_2)^2}.
 \ \in \Se(\{1:2\} = R[(x_1-x_2)^{\pm1}]
$$ 
Then Fact~\ref{fct:bichar-HM} says 
$r$ lifts to an $\FG$-invariant  $\Se(\{1:2\})$-valued bicharacter 
on $H_a(U(\frk{c}))$.
It can be written down as
\begin{align}
\label{eq:bichar:Heisenberg}
  r(D^{(i)} b^m \otimes  D^{(j)}b^n) 
  = \dfrac{\partial_{x_1}^i}{i!} \dfrac{\partial_{x_2}^j }{j!}
    \dfrac{m! c^m \delta_{m,n}}{(x_1-x_2)^2}.
\end{align}

\begin{lem}
Consider the twisting $V := T_*(\pi_0)^r = T_*(\FG(U(\frk{c})))^r$ 
of $T_*(\pi_0) = T_*(\FG(U(\frk{c})))$ 
by the bicharacter \eqref{eq:bichar:Heisenberg},
which is an $\RHS$-vertex algebra by Fact \ref{fct:Borcherds}.
Then in the ordinary vertex algebra associated to $V$ 
we have 
$$
 Y(b_{-1}v_0,x_1) Y(b_{-1}v_0,x_2) 
 = \dfrac{1}{(x-y)^2} + \nord{b(x_1)b(x_2)}.
$$
\end{lem}

Thus we have 

\begin{prop}
The ordinary vertex algebra associated to $V$ 
coincides with the Heisenberg vertex algebra.
\end{prop}

\subsection{Formal delta functions}

Here we recall the treatment of delta functions 
following \cite[Chap. 2]{K98}.
As before let us fix a commutative ring $R$.

Let us call elements of $R[[z_1^{\pm1},z_2^{\pm1},\ldots,z_n^{\pm1}]]$, that is, formal expressions
$$
 \sum_{m_1,m_2,\ldots,m_n \in \bb{Z}}
  a_{m_1,m_2,\ldots,m_n}z_1^{m_1} z_2^{m_2} \cdots z_n^{m_n},
$$ 
by ($R$-valued) formal distributions.
For a formal distribution $f(z) = \sum_{n \in \bb{Z}}f_n z^n$,
the residue is given by
$$
 \Res_z f(z) := f_{-1}. 
$$
It induces a non-degenerate pairing 
\begin{align}
\label{eq:res}
 \langle\; , \;\rangle: R[[z^{\pm1}]] \times R[z^{\pm1}] \longto R, \qquad
 \langle f,g \rangle := \Res_z \bigl(f(z)g(z)\bigr).
\end{align}

\begin{dfn}
The formal delta function $\delta(z,w)$ is the formal distribution
$$
 \delta(z,w):= \sum_{n \in \bb{Z}} z^{-n-1} w^n  
  \in R[[z^{\pm1},w^{\pm1}]].
$$
\end{dfn}

\begin{fct}
The formal delta function enjoys the following properties.
\begin{enumerate}
\item
For an arbitrary formal distribution $f(z) \in R[[z^{\pm1}]]$,
the product $f(z) \delta(z,w)$ is well-defined in $R[[z^{\pm1},w^{\pm1}]]$,
and one has
\begin{align}
\label{eq:delta:a}
 \Res_z f(z)\delta(z,w) = f(w).
\end{align}

\item
One has
\begin{align}
\label{eq:delta:b}
 \delta(z,w)=\delta(w,z).
\end{align}


\item
For $j \in \bb{Z}_{\ge0}$,
\begin{align}
\label{eq:delta:c}
(z-w)^{j+1}\partial_w^j\delta(z,w)=0
\end{align} 
holds 
in  $ R[[z^{\pm1},w^{\pm1}]]$.

\item
$\delta(z,t)\delta(w,t)$ and $\delta(w,t)\delta(z,t)$ 
are well defined in $R[[z^{\pm},w^{\pm},t^{\pm}]]$, and one has
\begin{align}
\label{eq:delta:d}
\delta(z,t)\delta(w,t) = \delta(w,t)\delta(z,t).
\end{align}
\end{enumerate}
\end{fct}

Using the pairing \eqref{eq:res} and the property \eqref{eq:delta:a},
one can show
\begin{align*}
\delta(z,w)f(z) = \delta(z,w)f(w)
\end{align*}
for any $f \in R[[z,w]]$.
Then replacing $a(z)$ by $\delta(z,t)$ and exchanging $t$ with $z$,
one obtains

\begin{cor}
\begin{align}
\label{eq:delta-delta}
\delta(z,t)\delta(w,t) = \delta(w,t)\delta(z,w).
\end{align}
\end{cor}

As a preliminary of the next subsection,
we introduce several well-known notions.

\begin{dfn}
\begin{enumerate}
\item
For an $R$-module $M$,
an $\End(M)$-valued formal distribution 
$ a(z) = \sum_{n \in \bb{Z}} a_n z^n \in \End(M)[[z^{\pm1}]]$
is called a field on $M$ if
for any $v \in V$ we have 
$a_j . v = 0$ for large enough $j$.


\item
For two distributions $f(z)$ and $g(z)$ in $R[[z^{\pm1}]]$,
we define the normal ordering by
$$
 \nord{f(z)g(w)} := f(z)_{+} g(w) + g(w) f(z)_{-},
$$
where for $f(z)=\sum_{n \in \bb{Z}} f_n z^n$ we used the symbols
$$
 f(z)_+ := \sum_{n \ge 0} f_n z^n, \qquad
 f(z)_- := \sum_{n < 0} f_n z^n.
$$
\item
For distributions $f_i(z)$ ($i =1,2,\ldots,m$) , we define 
the normal ordering by
$$
 \nord{f_1(z_1) f_2(z_2) \cdots f_m(z_m)} := 
 \nord{f_1(z_1) \nord{f_2(z_2) \cdots \nord{f_{m-1}(z_{m-1})f_m(z_m)}\cdots}}.
$$
\end{enumerate}
\end{dfn}

As is well-known, we have 

\begin{fct}
For two fields $a(z),b(z)$ on an $R$-module $M$,
the specialization $z=w$ of the normal ordering $\nord{a(z)b(w)}$,
that is,  $\nord{a(z)b(z)}$, 
is a well-defined field on $M$.
\end{fct}

Thus the following definition makes sense.

\begin{dfn}
For fields $f_i(z)$ ($i =1,2,\ldots,m$) 
on an $R$-module $M$, we define 
the (specialized) normal ordering by
$$
 \nord{f_1(z) f_2(z) \cdots f_m(z)} := 
 \nord{f_1(z) \nord{f_2(z) \cdots \nord{f_{m-1}(z)f_m(z)}\cdots}}.
$$
\end{dfn}

Now we have 

\begin{lem}
Consider an object $\Vd$ of $\Fun(\Fine,\RM)$ defined by
$$
 \Vd(I) := 
  \Se(I)[ \partial_{x_i}^n \delta(x_i,x_j) \mid 
         i \not\equiv j  \text{ in } I, \ n \in \bb{Z}_{\ge0}]
$$
for $I \in \Ob(\Fine)$ and
$$
 \Vd(f): \Vd(I) \longto \Vd(J), \quad 
 x_i \longmapsto x_{f(i)}
$$
for $f \in \Fine(I,J)$.
Here the multiplication of  $\partial_z^j \delta(z,w)$'s are 
given in terms of normal orderings,
and we assume that those normal orderings make sense.
Then $\Vd$ is a $T^*(H_a)$-module,
and also an $\Se$-module in $\Fun(\Fine,\RM,T^*(H_a))$.
Finally, $Vd$ is a singular commutative in $\Fun(\Fine,\RM,T^*(H_a),\Se)$,
that is, an $\AHS$-vertex algebra.
\end{lem}
\begin{proof}
Well-definedness as an object of $\Fun(\Fine,\RM)$ is easily checked.
The $T^*(H_a)$-module structure is given by delivation,
that is, $D^{(n)}_i \delta(x_i,x_j)= \partial_i^n \delta(x_i,x_j)/n!$.
(Although we used the fractional symbol $1/n!$,
 the coefficients are always in the commutative ring $R$.)
The $\Se$-module structure is obviously given.
The singular commutativity follows from \eqref{eq:delta:d}.
\end{proof}

Since $\Vd(\{1\}) = R$,
the associated ordinary vertex algebra is the trivial one.

\subsection{Vertex algebras of loop Lie algebras}

In \S\ref{subsec:affine} we construct the (ordinary) vertex algebras 
of affine Kac-Moody Lie algebras 
in the formulation of Borcherds reviewed in the previous subsections.
Before doing so,
we first construct the ordinary vertex algebras
of loop Lie algebras,
that is,
affine Lie algebra without the central extension.

Let $R$ be a fixed commutative ring containing $\bb{Q}$.

For a Lie algebra $\fg$ defined on $R$,
its universal enveloping algebra
is denoted by $U(\fg)$ as usual.
It is a cocommutative bialgebra 
with the comultiplication given by 
$\Delta(A)= A \otimes 1 + 1 \otimes A$ 
for $A \in \fg$.
By the Poincar\'{e}-Birkoff-Witt theorem,
$U(\fg)$ has a basis arising from a fixed 
totally ordered basis of $\fg$.
Hereafter we fix a  total order $\le$ 
on a basis of $\fg$.

The loop algebra of $\fg$ is an $R$-vector space
$$
 \Lg := \fg \otimes_R R[t^{\pm1}]
$$ 
with Lie algebra structure given by 
$$
 [A \otimes t^m,B \otimes t^n] := [A,B]\otimes t^{m+n}
$$
for $A,B \in \fg$ and $m,n \in \bb{Z}$.

Consider the one-dimensional trivial representation 
$R_0 = R v_0$ of $\fg \otimes R[t]$.
Here $v_0$ is the basis vector of $R_0$. 
The induced representation 
$$
 V_0(\fg) := \Ind_{\fg \otimes R[t]}^{\Lg} R_0 
    = U(\Lg) \otimes_{U(\fg \otimes R[t])} R_0
$$
is called the Verma module of $\Lg$.
$V_0(\fg)$ has a cocommutative bialgebra structure 
induced from that on $U(\Lg)$.

The Poincar\'{e}-Birkoff-Witt theorem 
gives the isomorphism
$$
  V_0(\fg) \simeq U(\fg \otimes t^{-1}R[t^{-1}])
$$
of $R$-vector spaces.
In particular, using a basis $\{a_i \mid i = 1,2,\ldots,\dim \fg\}$ 
of $\fg$ and denoting 
\begin{align}
\label{eq:J^a_n}
 J^a_n := a \otimes t^n,
\end{align}
for $a \in \fg$, we have a basis 
\begin{align}
\label{eq:basis}
 \bigl\{ J^{a_1}_{n_1} J^{a_2}_{n_2} \cdots J^{a_j}_{n_j} v_0 \mid
    j \in \bb{Z}_{\ge 0},\ 
    n_1 \le n_2 \le \cdots \le n_j <0,\  
   \text{ if } n_i = n_{i+1} \text{ then } a_i \le a_{i+1} \bigr\}
\end{align}
for $V_0(\fg)$.  

Let us recall the commutative cocommutative bialgebra $\FG = R[D^{(i)}]$ 
given in Definition \ref{dfn:FG}.
The action of $\FG$ on the polynomial ring as derivation
induces another action on $\Lg$.
Written explicitly, $\FG$ acts on $\Lg$ via 
$$
 D^{(i)}(A \otimes t^{-m}) = 
 \binom{m+i-1}{i} A \otimes t^{-m-i}.
$$
This action extends to $U(\Lg)$ and then restricts to $V_0(\fg)$.

\begin{rmk}
\label{rmk:FG-generate}
Under this $\FG$-action,
$V_0(\fg)$ is generated by $\{J^a_{-1} \mid a = 1,2,\ldots,\dim \fg\}$
over $R$.
\end{rmk}

Now we construct an $(\cA,\FG,\Se)$-vertex algebra $\VLg$ from $V_0(\fg)$.
The resulting ordinary vertex algebra (see Fact \ref{fct:Borcherds})
turns out to be the vertex algebra of affine Lie algebra with level $k=0$.

Recalling the functor $T_*$ in Definition \ref{dfn:T_*}, let us set 
$$
 \VLg := T_*\bigl( V_0(\fg) \bigr).
$$
In particular, 
$\VLg(\{1\})$ is the $R$-vector space $V_0(\fg)$, and 
$\VLg(\{1,2,\ldots,n\})$ is the $n$-th tensor product of $\VLg(\{1\})$.
By Example~\ref{eg:bialgebra:module} (2), 
$\VLg$ is an object of 
$\Fun(\Fine,\RM,T^*(\FG))$,
although it is not a ring object 
since the multiplicative structure on $ V_0(\fg)$ is not commutative.

We define the singular tensor product on $\VLg$ 
with the help of the trivial $\AHS$-vertex algebra $\Vd$ 
constructed in the previous subsection.
We will use the notation \eqref{eq:J^a_n} and \eqref{eq:basis} 
for elements of $\VLg(\{1\})=V_0(\fg)$.

\begin{lem}
\begin{enumerate}
\item
The $\fg$-valued distribution 
$$
  \partial_x^i J^a(x)
 := \partial_x^i \sum_{n \in \bb{Z}}J^a_{-n-1}x^n
 = a \otimes \partial_x^i\delta(t,x) 
\quad
\in U(\fg) \otimes \Vd(\{0:1\})
$$ 
with $a \in \fg$ and $i \in \bb{Z}_{\ge0}$
is a field on $V_0(\fg)$.
(Here we used $\{0:1\}$ to indicate the set of two elements with two equivalent classes,
 and the associated indeterminants are $t$ and $x$.)

\item
The correspondence 
$$
  J^{a_1}_{-n_1} J^{a_2}_{-n_2} \cdots J^{a_j}_{-n_j} v_0 
  \longmapsto
  \dfrac{1}{(n_1-1)!\cdots (n_j-1)!}
  \nord{ \partial_{x}^{n_1-1} J^{a_1}(x) \partial_{x}^{n_2-1} J^{a_2}(x)
          \cdots \partial_{x}^{n_j-1} J^{a_n}(x)}
$$
gives an isomorphism
$$
\theta: V_0(\fg) \longto U(\fg) \otimes \Vd(\{0:1\})
$$
of $R$-modules.

\item
$\theta$ extends to an isomorphism 
$$
 \theta: \VLg \longto T_*U(\fg) \otimes V_\delta
$$
of objects in $\Fun(\Fine,\RM,T^*(H_a))$.
In the right hand side $T_*U(\fg)$ is regarded as 
a trivial $T^*(H_a)$-module in $\Fun(\Fine,\RM)$.
\end{enumerate}
\end{lem}

\begin{proof}
The first part is well-known,
and the second part is 
obvious from the description of the basis on $V_0(\fg)$.
For the third part,
it is enough to notice that the isomorphism $\theta$ is 
equivalent with respect to the $H_a$-actions.
\end{proof}

A typical element of $\VLg(\{1,2\})$ is 
$J^a_{-1}v_0 \otimes J^b_{-1}v_0 = J^a(x_1) \otimes J^b(x_1)$,
and one of $\VLg(\{1:2\})$ is $J^a(x_1)J^b(x_2)$ 
Here the expression $J^a(x_1)J^b(x_2)$ means the product (or composition)
of fields on $V_0(\fg)$. 
The strict definition is given by the following lemma.

\begin{lem}
Define a bioperator 
$$
 \bullet: \VLg(\{1\}) \bigotimes_{\RM} \VLg(\{2\})
          \longto \VLg(\{1,2\})
$$ 
by
$$
 J^a_{-1}v_0 \bullet J^b_{-1}v_0
 := \theta^{-1}(\nord{J^a(x_1)J^b(x_2)})
$$
with $a,b \in \fg$.
Then it extends to a bioperator 
$$
 \bullet: \VLg \bigotimes_{\Fun(\Fine,\RM)} \VLg
          \longto \VLg
$$ 
\end{lem}

\begin{lem}
Define the singular tensor product 
$$
 \odot: \VLg(\{1\}) \bigotimes_{\RM} \VLg(\{2\})
          \longto \VLg(\{1:2\})= \VLg(\{1,2\})[(x_1-x_2)^{\pm1}]
$$
by
\begin{align}
\label{eq:stp}
 J^a_{-1}v_0 \odot J^b_{-1}v_0
 := \dfrac{1 \otimes J^{[a,b]}_{-1}v_0}{x_1-x_2} +  J^a_{-1}v_0 \bullet J^b_{-1}v_0
\end{align}
for $a,b \in \fg$.
Then it extens to the singular tensor product on $\VLg =  T_*\bigl( V_0(\fg) \bigr)$. 
\end{lem}

\begin{proof}
Recall that in the definition of the singular tensor product
we have the compatibility of $T^*(H_a)$-action and $\Se$-action.
Since $V_0(\fg)$ is generated by $J^a_{-1}$ under the $H_a$-action,
we immediately have the conclusion.
\end{proof}


\begin{rmk}
The singular tensor product $\odot$ gives 
the composition of fields on $V_0(\Lg)$.
It looks as 
$$
 J^a(z)\odot J^b(w) = \dfrac{J^{[a,b]}(w)}{z-w} + \nord{J^a(z)J^b(w)},
$$
which is the OPE usually used in calculations by physicists.
\end{rmk}

\begin{lem}
The singular tensor product on $\VLg$ is commutative,
so that $\VLg$ is an $(\RM,\FG,\Se)$-vertex algebra.
\end{lem}
\begin{proof}
By the definition of $\Se$,
it is sufficient to show 
$(x_1-x_2)^N v_1 \odot v_2 = (x_1-x_2)^N v_2 \odot v_1$
with some $N$ for any $v_1,v_2 \in \VLg(\{1\})$.
We demonstrate only for the case 
$v_1 = J^a_{-1}v_0$ and $v_2 = J^b_{-1}v_0$,
since the other cases follows by the $H_a$ action (as derivation) 
and the normal ordering 
(that is, by usual field calculs).
By the formula \eqref{eq:stp},  we have
\begin{align*}
&\theta(J^a_{-1}v_0 \odot  J^b_{-1}v_0 -  J^b_{-1}v_0 \odot  J^a_{-1}v_0)
= \Bigl(\dfrac{J^{[a,b]}(w)}{x_1-x_2} + \nord{J^a(x_1)J^b(x_2)}\Bigr) 
 - \Bigl(\dfrac{J^{[b,a]}(z)}{x_2-x_1} + \nord{J^a(x_2)J^b(x_1)}\Bigr) 
\\
&=\dfrac{J^{[a,b]}(x_2) - J^{[b,a]}(x_1)}{x_1-x_2}
 +[J^a(x_1)_+ , J^b(x_2)_+] + [J^b(x_2)_- , J^a(x_1)_-]
\\
&=\dfrac{J^{[a,b]}(x_2)   - J^{[b,a]}(x_1)}{x_1-x_2}
 -\dfrac{J^{[a,b]}(x_1)_+ - J^{[a,b]}(x_2)_+}{x_1-x_2}
 -\dfrac{J^{[b,a]}(x_2)_- - J^{[b,a]}(x_1)_-}{x_2-x_1}
\\
&=0.
\end{align*}
Thus we have the conclusion.
\end{proof}

Next we study the ordinary vertex algebra associated to $\VLg$.
Recall the proof of Fact \ref{fct:Borcherds},
in particular the construction of vertex operator $Y(A,x)$
using the Talor expansion formula \eqref{eq:taylor}.
For $A=J^a_{-1}v_0$, we can compute 
\begin{align*}
&Y(J^a_{-1}v_0,x_1)J^b_{-n}v_0
 =\sum_{i,j\ge0}x_1^i x_2^j 
   f_{*}\bigl(D_1^{(i)}D_2^{(j)}(J^a_{-1}v_0 \odot J^b_{-n}v_0)\bigr)
  \Bigr|_{x_2=0}
\\
&=\sum_{i,j\ge0}
   \binom{n+j-1}{j}x_1^ix_2^j f_{*}(J^a_{-1-i}v_0 \odot J^b_{-n-j}v_0)
  \Bigr|_{x_2=0}
\\
&=\sum_{i,j\ge0}\dfrac{1}{i!j!(n-1)!} x_1^ix_2^j 
   \theta^{-1}f_{*}\Bigl( 
    \partial_{x_1}^i\partial_{x_2}^{j+n-1} \dfrac{ J^{[a,b]}(x_2)}{x_1-x_2} +
    \nord{\partial_{x_1}^{i}J^a(x_1)\partial_{x_2}^{j+n}J^b(x_2)}
   \Bigr) 
  \Bigr|_{x_2=0}
\\
&=\sum_{i,j\ge0}\dfrac{1}{j!(n-1)!} x_1^ix_2^j 
  \theta^{-1}\Bigl( 
   \partial_{x_1}^{j+n-1} \dfrac{J^{[a,b]}(x_1)}{(x_1-x_2)^{i+1}}
   +\dfrac{1}{i!}\nord{\partial_{x_1}^{i}J^a(x_1)\partial_{x_1}^{j+n}J^b(x_1)}
  \Bigr) 
  \Bigr|_{x_2=0}
\\
&=\sum_{0\le k\le n-1}\dfrac{1}{(n-k-1)!} x_1^{-k-1} 
   \theta^{-1}\Bigl( \partial_{x_1}^{n-k-1} J^{[a,b]}(x_1) \Bigr)
  +\sum_{i\ge0}\dfrac{1}{i!(n-1)!} x_1^{-j} 
   \theta^{-1}\Bigl( 
     \nord{\partial_{x_1}^{i}J^a(x_1)\partial_{x_1}^{n}J^b(x_1)}
    \Bigr) 
\\
&=\sum_{i<0}  x_1^{i} J^{[a,b]}_{-i-n-1}v_0
 +\sum_{i\ge0} x_1^{i} J^a_{-i-1}J^b_{-n}v_0.
\\
\end{align*}

On the other hand,
in the ordinary vertex algebra of loop Lie algebra, 
one associates to $J^{a}_{-1}v_0$ the field $J^a(z)$,
which acts on $J^b_{-n}v_0$ ($n>0$) as
\begin{align*}
 J^{a}(z)J^b_{-n}v_0 
& =\sum_{i\ge -n}z^i J^a_{-i-1}J^b_{-n} v_0 
\\
&=\sum_{i<0}  z^{i} J^{[a,b]}_{-i-n-1}v_0
 +\sum_{i\ge0} z^{i} J^a_{-i-1}J^b_{-n}v_0.
\end{align*}
Thus we have
$$
 Y(J^a_{-1}v_0,z)J^b_nv_0 =  J^{a}(z) J^b_nv_0.
$$
Similarly we have 
$$
 Y(J^a_{-1}v_0,z)A =  J^{a}(z) A
$$
for any $A  \in V_0(\fg)$.
Then by Goddard's uniqueness Theorem 
we have 
$$
 Y(J^a_{-1}v_0,z) =  J^{a}(z)
$$
as fields.
Finally by the reconstruction theorem \cite[2.3.11]{FB},
we conclude

\begin{prop}
The ordinary vertex algebra structure on $\VLg(\{1\})$ 
coincides with the ordinary vertex algebra $V_0(\fg)$.
\end{prop}

\subsection{Vertex algebras of affine Kac-Moody  Lie algebras}
\label{subsec:affine}




Let us return to the $(\RM,\FG,\Se)$-vertex algebra $\VLg$
constructed from the Verma module $V_0(\fg)$ of the loop algebra $\Lg$. 
We will use the twisting operation reviewed in \S\ref{subsec:twist}
to construct another $(\RM,\FG,\Se)$-vertex algebra 
such that it corresponds to the ordinary vertex algebra $V_k(\fg)$ 
of affine Kac-Moody Lie algebra with arbitrary level $k$.

Let $R$ be a commutative field again.
Let $k$ be an arbitrary element of $R$,
which will be the level of affine Lie algebra $\ghat$.
Let us fix an invariant symmetric bilinear form on $\fg$
and denote it by $(\;,\;)$.

\begin{dfn}
Let $r$ be an $\FG$-invariant $S(\{1:2\})$-valued 
$R$-bicharacter of $V_0(\fg)$ such that 
\begin{align}
\label{eq:bichar:al}
 r(J^a_{-1}v_0 \otimes J^b_{-1}v_0) =  \dfrac{k(J^a,J^b)}{(x_1-x_2)^2}.
\end{align}
The $\FG$-invariant bicharacter $r$ is uniquely determined from this formula
since $V_0(\fg)$ is $\FG$-generated by 
$\{J^a_{-1}  \mid a=1,2,\ldots,\dim \fg\}$,
as we noted in Remark~\ref{rmk:FG-generate}.
\end{dfn}

Then by Fact \ref{fct:twisting}, 
the twisting of the $(\RM,\FG,\Se)$-vertex algebra $\VLg$ 
by the singular bicharacter defined by \eqref{eq:bichar:al}
is another $(\RM,\FG,\Se)$-vertex algebra.
Let us denote this new one by $\Vghat$.
$\Vghat(\{1\})$ is an $R$-vector space 
with a basis \eqref{eq:basis}.
To distinguish it from the old $\VLg(\{1\})$, 
let us denote the vacuum vector in the new one by $v_k$,
and denote the basis as 
\begin{align}
\label{eq:basis:k}
 \bigl\{ J^{a_1}_{n_1} J^{a_2}_{n_2} \cdots J^{a_l}_{n_l} v_k \mid
    l \in \bb{Z}_{\ge 0},\ 
    n_1 \le n_2 \le \cdots \le n_l <0,\  
   \text{ if } n_i = n_{i+1} \text{ then } a_i \le a_{i+1} \bigr\}.
\end{align}

Recalling the formula \eqref{eq:stp} 
for the singular tensor product in $\Vghat$ 
and the twisted product \eqref{eq:twisted_product} yields
$$
 \bigl[Y(J^a_{-1}v_k,z),Y(J^b_{-1}v_k,w)\bigr]= 
 \dfrac{k(J^a,J^b)}{(z-w)^2} + \dfrac{[J^a,J^b](w)}{z-w},
$$
which coincides with the formula in the vertex algebra 
of affine Kac-Moody Lie algebra $\ghat$ with level $k$.
Therefore we get

\begin{prop}
For the twisting $\Vghat$ 
of the $(\RM,\FG,\Se)$-vertex algebra $V$ 
by the singular bicharacter defined by \eqref{eq:bichar:al},
the associated ordinary vertex algebra $\Vghat(\{i\})$ 
coincides with the vertex algebra $V_k(\fg)$ 
of Kac-Moody Lie algebra $\ghat$ with level $k$.
\end{prop}

Applying our construction to the case 
where $\fg$ is the one-dimensional commutative Lie algebra,
one gets as $\Vghat(\{1\})$ the Heisenberg vertex algebra
(denoted as $\pi_0$ in \cite[\S\S 2.1 -- 2.4]{FB}).
Similarly from the positive-definite even lattice 
one gets the lattice vertex algebra.
These two cases were investigated in \cite{P}.

\subsection{The Virasoro vertex algebra}

The ordinary vertex algebra attached to Virasoro algebra 
can also be treated in our formulation.
Let us denote by $Vir$ the Virasoro Lie algebra 
with generators $\{L_n \mid n \in \bb{Z} \}$ 
and the central element $C$ 
defined over the complex number field $\bb{C}$.
The commutation relation is given by 
\begin{align}
\label{eq:vir}
 [L_m,L_n]=(m-n)L_{m+n} + \dfrac{ m^3-m }{12}  \delta_{m+n,0} C
\end{align}
as usual.

We construct an $(\RM,\FG,\Se)$-vertex algebra $\Vvir$ 
(with $R = \bb{C}$)
as follows.

Fix a complex number $c \in \bb{C}$.
Consider a Lie subalgebra 
$Vir_+ := \bigoplus_{n \in \bb{Z}_{\ge -1}}\bb{C} L_n \oplus \bb{C} C$ 
of $Vir$ and its one-dimensional representation 
$\bb{C}_c = \bb{C} v_c$ 
where $L_n$'s act trivially and $C$ acts by $c$.
Denote the induced representation of $Vir$ by
$$
 \Vir_c := \Ind^{Vir}_{Vir_+} \bb{C}_c
        = U(Vir) \otimes_{U(Vir_+)} \bb{C}_c.
$$
$U(Vir)$ has a cocommutative bialgebra structure,
and it induces another structure on $\Vir_c$.
In particular, the comultiplication on $L_nv_c$ ($n<-1$) 
is given by 
$\Delta(L_n v_0)= (L_n \otimes 1 + 1 \otimes L_n)v_c \otimes v_c$. 

We apply the construction of $\VLg$ to the derived algebra 
$[Vir,Vir]$, i.e, the Virasoro Lie algebra without central extension.
The space of fields, that is the $\bb{C}$-vector space $\Vvir(\{1\})$, 
is given by $\Vir_0$.
It has a basis 
$$
 \bigl\{L_{n_1} L_{n_2} \cdots L_{n_l} v_0 \mid
    l \in \bb{Z}_{\ge 0},\ 
    n_1 \le n_2 \le \cdots \le n_l <-2  \bigr\}.
$$

The action of the cocommutative bialgebra 
$\FG = R[D^{(i)}]$ on $\Vir_0$
is given by
$$
 D^{(i)} A :=  \dfrac{1}{i!}L_{-1}^i A
$$
for $A \in \Vir_0$.
By the commutation relation \eqref{eq:vir} 
one can check the formula
\begin{align}
\label{eq:vir:H-action}
D^{(i)} L_{-n} v_0 =  \binom{n+i-1}{i} L_{-n-i}v_0
\end{align} 
for $n\in \bb{Z}_{>1}$ and $i \in \bb{Z}_{\ge0}$.

As $\VLg$, we get an $(\RM,\FG,\Se)$-vertex algebra $V$.
Then we want to take a twist of $V$ by some singular bicharacter.
Consider the $\FG$-invariant $\Se(\{1:2\})$-bicharacter $r$ 
of $\Vir_0$ such that 
$$
 r(L_{-2} \otimes L_{-2}) = \dfrac{c/2}{(x_1-x_2)^4}.
$$
This formula determines $r$ uniquely,
since $V_0$ is $\FG$-generated by $L_-1$ 
by the action \eqref{eq:vir:H-action}.
Then Lemma \ref{lem:sing-bichar} says that 
there is a singular bicharacter on $V$.
By Fact \ref{fct:twisting} 
we have a twisted $\RHS$-vertex algebra, 
which is denoted by $\Vvir$.

As in $\Vghat$, we rename the vacuum vector of $\Vvir(\{1\})$ 
as $v_c$. 
Then one can check that $T(z) :=  Y(L_{-2}v_c,z)$ satisfies
$$
 \bigl[ T(z),T(w) \bigr]
 = \dfrac{ 2T(w)}{(z-w)^2}
  + \dfrac{\partial_w T(w)}{z-w}
  + \dfrac{c/2}{(z-w)^4},
$$ 
which coincides with the OPE of conformal field of
the Virasoro vertex algebra.
Finally we have 

\begin{prop}
For the $\RHS$-vertex algebra $\Vvir$,
the associated ordinary vertex algebra $\Vvir(\{1\})$
coincides with the Virasoro vertex algebra with central charge $c$. 
\end{prop}


\section{Quantum vertex algebra }

In this section we follow the formulation of $\AHS$-quantum vertex algebras
given in \cite{B:2001}.
It can be considered as a deformation of $\AHS$-vertex algebras
discussed in the previous section.
We remark that there are several formulations on 
deformation of ordinary vertex algebras,
for example \cite{FR}, \cite{EK}, \cite{Li:2005}, \cite{AB}, \cite{Li:2010} and \cite{Li:2011}.

\subsection{Borcherds' formulation}

We begin with introduction of braided rings.
Let $C$ be a fixed commutative ring.

\begin{dfn}
\label{dfn:Rmatrix}
Let $A$ be a unital associative $C$-algebra.
A $C$-homomorphism $R: A \otimes A \to A \otimes A$ is called an $R$-matrix
if it satisfies the following conditions:
\begin{enumerate}
\item Yang-Baxter equation:
 $R_{12}R_{13}R_{23} = R_{23}R_{13}R_{12}$.
\item
 $R_{12}m_{12} = m_{12}R_{23}R_{13}$ and  
 $R_{12}m_{23} = m_{23}R_{12}R_{13}$ 
 as $C$-homomorphisms $A \otimes A \otimes A \to A \otimes A$.
\item
 $R(1\otimes a) = 1 \otimes a$ and 
 $R(a\otimes 1) = a \otimes 1$ for any $a \in A$.
\end{enumerate}
\end{dfn}

The following lemma is due to \cite[Lemma 10.1]{B:1998}
where the ring $A$ is assumed to be commutative.

\begin{lem}
Suppose $A$ is a unital associative ring and 
$R$ is an $R$-matrix for $A$.
Then the operation $m_{12}R_{12}$ 
defines another unital associative ring 
$(A,m_{12}R_{12},1_A)$,
where $1_A$ is the unit of the original ring structure on $A$.
\end{lem}

\begin{proof}
We only write down the proof of the associativity.
\begin{align*}
 m_{12}R_{12}m_{23}R_{23}
=m_{12}m_{23}R_{12}R_{13}R_{23}
=m_{12}m_{23}R_{23}R_{13}R_{12}
=m_{12}m_{12}R_{23}R_{13}R_{12}
=m_{12}R_{12}m_{12}R_{12}.
\end{align*}
\end{proof}

\begin{dfn}
A braided ring $A$ is a ring with an $R$-matrix $R$ such that
\begin{align}
\label{eq:braid-comm}
 m_A R = m_A \tau: A \otimes A \longto A. 
\end{align}
Here $\tau: a \otimes b \mapsto b \otimes a$ 
is the twist map and $m_A$ is the multiplication of $A$.
\end{dfn}

The  twisting construction gives us a family of braided rings. 
Before stating Borchreds' construction,
let us note
\begin{rmk}
\begin{enumerate}
\item
For a cocommutative bialgebra $B$,
$C$-valued bicharcters of $B$ 
form a commutative monoid under the multiplication 
$$
 (r*s)(a \otimes b) := \sum_{(a),(b)} r(a' \otimes b')s(a'' \otimes b'')
$$
and the unit 
\begin{align}
\label{eq:bichar:unit}
 \ve(a \otimes b) := \ve(a) \ve(b)
\end{align}
given by the counit $\ve$ of $B$.
A bicharacter $r$ is called invertible 
if it has its inverse $r^{-1}$ in this monoid.

\item
If $B$ is a Hopf algebra, 
then the inverse $r^{-1}$ is given by 
$$
 r^{-1}(a \otimes b) = r(S_B(a) \otimes b)
$$
with $S_B$ the antipode of $B$.
\end{enumerate}
\end{rmk}

The followin lemma is due to \cite[Lemma/Definition 2.6]{B:2001},
where $M$ is assumed to be commutative.

\begin{lem}
Consider the twisting $\wt{M}$ 
of the commutative cocommutative bialgebra $M$ 
by a $C$-valued bicharacter $r$.
If $r$ is invertible, then $\wt{M}$ is a braided ring.
\end{lem}
\begin{proof}
We wride down a proof for the sake of completeness.
The $R$-matrix for the braided ring $\wt{M}$ is given by 
$$
 R(a \otimes b) := 
 \sum_{(a),(b)} a' \otimes b' r'(b'' \otimes a''),
$$
with $r'$ a bicharacter defined to be 
$$
 r'(a \otimes b) := 
 \sum_{(a),(b)}r(a' \otimes b') r^{-1}(b'' \otimes a'').
$$ 

We show that the above formula does define an $R$-matrix.
For the Yang-Baxter equation, we have
\begin{align*}
  R_{12}R_{13}R_{23}(a \otimes b \otimes c)
&=R_{12}R_{13}\bigl(\sum a \otimes b^{(1)} \otimes c^{(1)} r'(c^{(2)}\otimes b^{(2)})\bigr)
\\
&=R_{12}\bigl(\sum a^{(1)} \otimes b^{(1)} \otimes c^{(1)} 
             r'(c^{(2)}\otimes a^{(2)}) r'(c^{(3)}\otimes b^{(2)})\bigr)
\\
&=\sum a^{(1)} \otimes b^{(1)} \otimes c^{(1)} 
             r'(b^{(2)}\otimes a^{(2)}) r'(c^{(2)}\otimes a^{(3)}) r'(c^{(3)}\otimes b^{(3)})
\end{align*}
and
\begin{align*}
  R_{23}R_{13}R_{12}(a \otimes b \otimes c)
&=R_{23}R_{13}\bigl(\sum a \otimes b^{(1)} \otimes c^{(1)} r'(b^{(2)}\otimes a^{(2)})\bigr)
\\
&=R_{23}\bigl(\sum a^{(1)} \otimes b^{(1)} \otimes c^{(1)} 
             r'(c^{(2)}\otimes a^{(2)}) r'(b^{(2)}\otimes a^{(3)})\bigr)
\\
&=\sum a^{(1)} \otimes b^{(1)} \otimes c^{(1)} 
             r'(c^{(2)}\otimes b^{(2)}) r'(c^{(3)}\otimes a^{(2)}) r'(b^{(3)}\otimes a^{(3)}).
\end{align*}
These two equations are equal by the cocommutativity of $M$.

The first half of the second condition in Definition \ref{dfn:Rmatrix} can be checked by
\begin{align*}
R_{12}m_{12}(a \otimes b \otimes c)
&=R_{12}(ab \otimes c)
 =\sum a^{(1)}b^{(1)} \otimes c^{(1)} r'(c^{(2)} \otimes a^{(2)} b^{(2)})
\\
& =\sum a^{(1)}b^{(1)} \otimes c^{(1)} r'(c^{(2)} \otimes a^{(2)})
    r'(c^{(3)}\otimes b^{(2)})
\end{align*}
and
\begin{align*}
m_{12}R_{23}R_{13}(a \otimes b \otimes c)
&=m_{12}R_{23}\bigl(\sum a^{(1)} \otimes b \otimes c^{(1)} r'(c^{(2)}\otimes a^{(2)})\bigr)
\\
&=m_{12}\bigl(\sum a^{(1)} \otimes b^{(1)} \otimes c^{(1)} 
             r'(c^{(2)}\otimes b^{(2)}) r'(c^{(3)}\otimes a^{(2)})\bigr)
\\
&=\sum a^{(1)} b^{(1)} \otimes c^{(1)} 
             r'(c^{(2)}\otimes b^{(2)}) r'(c^{(3)}\otimes a^{(2)}).
\end{align*}
We used the cocommutativity and the bialgebra property in this demonstration.
The last half is shown similarly.

The third consition in Definition \ref{dfn:Rmatrix}
is easily checked.
Note that we have not used the commutativity of $M$ so far.

The braided commutativity condition \eqref{eq:braid-comm} 
can be checked by
\begin{align*}
\wt{m}R(a \otimes b)
&=\wt{m}\bigl(\sum a^{(1)}\otimes b^{(1)}r'(b^{(2)}\otimes a^{(2)})\bigr)
\\
&=\sum a^{(1)} b^{(1)}r(a^{(2)}\otimes b^{(2)}) r'(b^{(3)}\otimes a^{(3)})
\\
&=\sum a^{(1)} b^{(1)}r(b^{(2)}\otimes a^{(2)})
\\
&=\wt{m}\tau(a \otimes b),
\end{align*}
where $\wt{m}$ is the twisted multiplication on $\wt{M}$,
and at the last line we used the commutativity of $M$.
\end{proof}

\begin{rmk}
The twisting $\wt{M}$ by the unit bicharacter \eqref{eq:bichar:unit} 
is the original algebra $M$.
In this case, $r'=\ve$ and $R$ is the identity operator.
\end{rmk}

The notion of $R$-matrix can also be introduced 
in an additive symmetric monoidal category $\cA$.
Hereafter we switch to this categorical setting.
Using the singular tensor product, Borcherds introduced  

\begin{dfn}[{\cite{B:2001}}]
Let $H$ be a cocommutative bialgebra object in $\cA$,
and $S$ be a commutative ring object in 
the additive symmetric monoidal category $\Fun(\Fine,\cA,T^*(H))$.
Define a quantum $\AHS$-vertex algebra 
to be a singular braided ring in $\Fun(\Fin,\cA,T^*(H),S)$.
\end{dfn}

The twisting construction gives some examples of 
quantum $\AHS$-vertex algebra.
The main theorem in \cite{B:2001} was

\begin{fct}[{\cite[Theorem 4.2]{B:2001}}]
\label{fct:twisting:quantum}
Suppose that $H$ is a cocommutative bialgebra in $\cA$ 
and that $S$ is a commutative ring in $\Fun(\Fine,\cA,T^*(H))$.
Assume that $r$ is an invertible $S(\{1:2\})$-valued bicharacter 
of a commutative and cocommutative bialgebra $M$ in $\cA$.
Then the twisting of $T_*(M)$ by $r$
is a quantum $(\cA,H,S)$-vertex algebra.
\end{fct}

\subsection{Yangian}

To construct ordinary vertex algebras in the framework of Borcherds,
the commutative ring object $\Se$ (Definition~\ref{dfn:Se})
in $\Fun(\Fine,\RM,T^*(\FG))$ was a key ingredient.
It encodes the singular behavior of vertex operators 
$Y(\;,z)$ in vertex algebras.

In this subsection we consider another singular data.
Fix an element $t \in R$.

\begin{dfn}
Define an object $\SY$ in $\Fun(\Fine,\RM)$ by 
\begin{align}
\label{eq:SY}
 \SY(I) := R[(x_i - x_j - n t)^{\pm1} \mid i\not\equiv j \text{ in } I, \ n \in \bb{Z}]
\end{align}
for $I \in \Ob(\Fine)$,
and 
$$
 \SY(f):\SY(I) \longto \SY(J), \quad 
 (x_i -x_j- n t) \longmapsto (x_{f(i)}-x_{f(j)}- n t)
$$ 
for $f \in \Fine(I,J)$.
\end{dfn}

$\SY$ with $t=0$ is nothing but $\Se$.
Similarly as Lemma \ref{lem:Se}, one can check

\begin{lem}
$\SY$ is a commutative ring object
in $\Fun(\Fine,\RM,T^*(\FG))$,
where the action of $T^*(\FG)$ on $\SY$ is given by 
the derivation.
\end{lem}

Thus we can consider a quantum $\RHS$-vertex algebra.
Yangians (precisely speaking, the algebras of Drinfeld currents of Yangian) 
is an example of this setting.

\subsection{Deformed chiral algebras}

In \cite{FR} Frenkel and Reshetikhin introduced the notion of 
deformed chiral algebras, in order to formulate the deformation of 
ordinary vertex algebras and treat the deformed $W$-algebras 
which emerged in the mid 1990s.

\begin{dfn}
A deformed chiral algebra is a collection of the following data:
\begin{itemize}
\item
 A $\bb{C}$-vector space $V$ called the space of fields.
\item
 A $\bb{C}$-vector space $W = \cup_{n \ge 0} W_n$ 
 called the space of states, which is union of finite dimensional
 subspaces $W_n$. 
 We consider a topology on $W$ in which $\{W_n\mid n \ge 0\}$ is
 the base of open neighborhoods of $0$.
\item
  A linear map $Y : V \to \End(W) \widehat{\otimes} [[z, z^{-1}]]$ 
  such that for each $A \in V$ each linear operator $A_n\in\End(W)$ in the expansion 
  $Y(A, z) = \sum_{n \in \bb{Z}} A_n z^{-n}$ satisfies $A_n W_m \subset W_{m+N(n)}$ 
  for any $m \in \bb{Z}_{\ge0}$ with some $N(n) \in \bb{Z}$ depending only on $A$.
\item
 A meromorphic function $S(x) : \bb{C}^\times  \to \Aut(V \otimes V)$, 
 satisfying the Yang-Baxter equation 
\begin{equation}
\label{eq:YBE}
 S_{12}(z)S_{13}(z w)S_{23}(w) = S_{23}(w)S_{13}(z w)S_{12}(z) 
\end{equation}
 for any $z,w \in \bb{C}^\times$.
\item
 A lattice $L \subset \bb{C}^\times$, 
 which contains the poles of $S(x)$.
\item
 An element $\Omega \in V$ such that $Y(\Omega,z) = \id$.
\end{itemize}
These data should satisfy the following axioms:
\begin{enumerate}
\item
For any $A_i \in V$ ($i=1,\ldots,n$),
the composition $Y(A_1,z_1) \cdots Y(A_n,z_n)$ converges
in the domain $|z_1| \gg \cdots \gg |z_n|$ 
and can be continued to a meromorphic operator valued function
$$
 R(Y(A_1, z_1) \cdots Y(A_n, z_n)): (\bb{C}^\times)^n \to \Hom(W,\overline{W}),
$$
where $\overline{W}$ is the completion of $W$ with respect to its topology.

\item
Denote $R(Y(A,z)Y(B,w))$ by $Y(A\otimes  B;z,w)$. 
Then
$$ 
 Y(A \otimes B; z,w) = Y(S(w/z)(B \otimes A);w,z).
$$

\item
The poles of the meromorphic function $R(Y(A,z)Y(B,w))$ lie on the lines 
$z = w\gamma$ with $\gamma \in L$. 
For each such line and $n \in\bb{Z}_{\ge0}$, 
there exists $C_n \in V$ such that 
$$
 \Res_{z=w\gamma} R(Y(A,z)Y(B,w))(z-w\gamma)^n \dfrac{dz}{z}
= Y (C_n,w).
$$
\end{enumerate}
\end{dfn}

Let us relate the deformed chiral algebra $(V,W,Y,S(x),L,\Omega)$ 
with a quantum $\AHS$-vertex algebra.
We begin with the singular data $S$ for the deformed chiral algebra.

\begin{dfn}
For a lattice $L\subset \bb{C}^\times$. 
Define an object $S_{L}$ in $\Fun(\Fine,\RM)$ by 
\begin{align*}
 S_L(I) := R[(x_i/x_j - \gamma)^{\pm1} \mid i\not\equiv j \text{ in } I, \ \gamma \in L]
\end{align*}
for $I \in \Ob(\Fine)$,
and 
$$
 S_L(f):\SY(I) \longto \SY(J), \quad 
 (x_i/x_j- \gamma) \longmapsto (x_{f(i)}/x_{f(j)}- \gamma)
$$ 
for $f \in \Fine(I,J)$.
\end{dfn}

Next we need a formal group ring.

\begin{dfn}
Let $H_m$ be the formal group ring of the one-dimensional multiplicative gormal group
(corresponding to the formal group law $F(X,Y)=X Y$).
\end{dfn}

As in the case of $H_a$, 
one can consider the action of $T^*(H_m)$ on $S_L$ 
(by difference operators preserving $L$).
Then one can show

\begin{lem}
$S_L$ is a commutative ring object in $\Fun(\Fine,\RM,T^*(H_m))$.
\end{lem}

Thus we can consider a quantum $\MHS$-vertex algebra.
Our result is 

\begin{thm}
Let $V_1$ be a $\bb{C}$-vector space and 
$S(x):\bb{C}^\times  \to \Aut(V_1 \otimes V_1)$ be a meromorphic function
satisfying the Yang-Baxter equation \eqref{eq:YBE}.
Let $V$ be a quantum $(\CM,H_M,S_L)$-vertex algebra 
given by the twisting using $S(x)$ 
(so that the underlying vector space of $V(\{1\})$ is $V_1$).
Then  $V(\{1\})$ has a structure of deformed chiral algebra.
\end{thm}

The proof is similar as in the case of non-quantum $\RHS$-vertex algebras,
so we omit it.

\section{Chiral algebras}

Let us recall the formulation of chiral algebras due to Beilinson and Drinfeld \cite{BD}.
We will use the notion of factorization algebra, 
which is equivalent to the chiral algebra (in the case of smooth algebraic curves).
 
For an algebraic curve $X$ defined over some field $k$
and an object $I$ in $\Fin$,
$X^I$ denotes the symmetric product over $k$.
$\QCoh(X^I)$ denotes the category of quasi-coherent sheaves on $X^I$.

\begin{dfn}
Let $X$ be a smooth algebraic curve defined over $\bb{C}$.
A factorization algebra over $X$ consists of data 
$\{F_I \in \Ob\QCoh(X^I) \mid I \in \Ob\Fin \}$
such that 
\begin{enumerate}
\item
$F_I(\Delta) =0$,
where $\Delta$ is the (big) diagonal of $X^I$.

\item
$\Delta_{J/I}^* F_J \simto F_J$ for $p: J \twoheadrightarrow I$,
where $\Delta_{J/I}: X^I \hookrightarrow X^J$ 
is the natural inclusion morphism induced by $p$.

\item
$j_{J/I}^* F_J \simeq j_{J/I}^*\left(\boxtimes_{i \in I} F_{p^{-1}(i)}\right)$ 
for $p: J \twoheadrightarrow I$,
where $j_{J/I} : U^{J/I} \hookrightarrow X^J$ is the inclusion 
morphism from 
$U^{J/I}:=\{(x_j) \in X^J \mid x_j \neq x_{j'} \text{ if } p(j) \neq p(j') \}$
to $X^J$.

\item
There exists $1 \in F_1(X)$ such that 
for any $f \in F_1(U)$ (where $U \subset X$ is an arbitrary open subscheme) 
the element $1 \boxtimes f \in F_2(U^2\setminus \Delta)$ extends
across $\Delta$ and restricts to $f \in F_1(U) \simeq F_2(\Delta|_U)$. 
\end{enumerate}
\end{dfn}

A morphism between factorization algebras can be defined naturally.
One of the fundamental results in the Beilinson-Drinfeld theory is

\begin{fct}\label{fct:BD}
There exists an equivalence of categories between 
the category of quasi-conformal ordinary vertex algebras $V$ 
and the category of factorization algebras $\{F_I\}$ 
such that $F_1 = \Aut_X \times_{\Aut_{\cal{O}_X}} V$.
\end{fct}

Here we used the term \emph{quasi-conformal} in the meaning of \cite[\S 6.2]{FB}.
Let us recall its definition briefly.
The space $\cal{O} := \bb{C}[[z]]$ of formal series of one variable 
with complex coefficients may be considered as 
a complete topological $\bb{C}$-algebra 
(with the topology given by the  unique maximal ideal).
Let us also consider the Lie algebras 
$$
  \Der \cal{O} := \bb{C}[[z]]\partial_z
  \  \supset \ 
  \Der_0 \cal{O} := z\bb{C}[[z]]\partial_z
  \  \supset \ 
  \Der_+ \cal{O} := z^2\bb{C}[[z]]\partial_z.
$$
Let us denote by $L_n := -z^{n+1}\partial_z \in \Der\cal{O}$ 
for $n\in\bb{Z}_{\ge -1}$.

\begin{dfn}
An ordinary vertex algebra is called quasi-conformal if it has an action of $\Der \cal{O}$ 
such that 
\begin{itemize}
\item
the formula
$$
 \Bigl[\sum_{n\ge -1}v_n L_n, Y(A,z)\Bigr]
=\sum_{m\ge-1}\dfrac{1}{(m+1)!}(\partial_w^{m+1} v(z))Y(L_m A,z)
$$
holds for any $A \in V$ and 
any $v(z)\partial_z = \sum_{n\ge-1} v_n z^{n+1}\partial_z\in \Der\cal{O}$,
\item 
the element $L_{-1} = -\partial_z$ acts as the translation operator $T$,
\item
$L_0 = -z\partial_z$ acts semisimply with integral eigenvalues, 
\item
the Lie subalgebra $\Der_+ \cal{O}$ acts locally nilpotently.
\end{itemize}
\end{dfn}

A conformal ordinary vertex algebra (ordinary vertex algebra with a Virasoro element) 
is the canonical example of quasi-conformal ordinary vertex algebra.
Let us also mention that $\Lie(\Aut\cal{O}) = \Der_0\cal{O}$,
where $\Aut \cal{O}$ is the group of continuous automorphisms of $\cal{O}$.
The axiom of quasi-conformal ordinary vertex algebra says that $\Aut\cal{O}$.
acts on the vertex algebra.
Since $\Aut \cal{O}$ is the infinitesimal symmetry of an algebraic curve,
the appearance of quasi-conformal ordinary vertex algebra in Fact \ref{fct:BD} is natural.

\begin{thm}
Let us consider the $(\cal{A},H,S)$-vertex algebra with the setting 
$$
 \cal{A} = \QCoh(X),\quad 
 H = \Der \cal{O}_X,\quad
 S(J)=\cal{O}_{U^{J/I}}.
$$
In the definition of $S(J)$ for $J \in \Ob \Fine$, 
$I$ is uniquely determined by 
the surjection 
$J \twoheadrightarrow I$ 
corresponding to $J$.

Then the $(\cal{A},H,S)$-vertex algebra has a structure of factorization algebra,
and the associated vertex algebra is quasi-conformal.
\end{thm}

The proof is similar as in the case of non-quantum $\RHS$-vertex algebras,
so we omit it.


\end{document}